\documentclass[final]{siamltex}

\usepackage{mathrsfs}
\usepackage{color}
\usepackage{graphicx, epstopdf}
\usepackage{amsmath}
\usepackage{amsfonts}
\usepackage{amssymb}
\usepackage{textcomp}
\usepackage{latexsym}

\usepackage{threeparttable}

\newtheorem{remark}[theorem]{Remark}

\newcommand\ddelta\bigtriangledown
\newcommand\ld\lambda
\newcommand\Ld\Lambda

\def \bP {\Bbb P}

\def \bZ {\Bbb Z}

\title{Superconvergence of Local Discontinuous Galerkin method for
one-dimensional linear parabolic equations\thanks{The second author
was supported in part by the US National Science Foundation through
grant DMS-1115530.}}
\author{Waixiang Cao \footnotemark[2]\ \footnotemark[3]
\and Zhimin Zhang \footnotemark[2]\ \footnotemark[4]}

\begin{document}
\maketitle
\renewcommand{\thefootnote}{\fnsymbol{footnote}}
\footnotetext[2] {Beijing Computational Science Research Center,
Beijing, 100084, China.}
 \footnotetext[3]{ College of Mathematics and Computational Science and Guangdong Province Key Laboratory of Computational Science, Sun Yat-sen
          University, Guangzhou, 510275, China. }
 \footnotetext[4]{ Department of Mathematics, Wayne State
University, Detroit, MI 48202, USA.}


\begin{abstract}
  In this paper, we study superconvergence properties of the
   local discontinuous Galerkin method for one-dimensional linear
   parabolic equations when alternating fluxes are used. We prove,
   for any polynomial degree $k$, that the numerical fluxes  converge at a rate of
   $2k+1$ (or $2k+1/2$) for all mesh nodes and the domain average under some suitable initial discretization.
    We further prove a $k+1$th superconvergence rate for the derivative approximation
   and a $k+2$th superconvergence rate for the function value approximation at the Radau points.
    Numerical experiments demonstrate that in most cases, our error estimates are optimal,
   i.e., the error bounds are sharp.
\end{abstract}



\section{ Introduction}

   The superconvergence behavior  of discontinuous Galerkin (DG) and
   local discontinuous Galerkin (LDG) \cite{Cockburn;Shu:SIAM1998} methods has been
   studied for some years.  Some early results can be found in Thom\'{e}e's 1997 book \cite{thomee}.
   Later, in \cite{Adjerid;Massey2006}, Adjerid et al. showed
   a $k+2$th superconvergence rate of the DG solution at the
   downwind-biased Radau points for  some ordinary differential equations; in
   \cite{Celiker.Cockburn}, Celiker and Cockburn studied
   superconvergence of the numerical traces for  DG and hybridizable DG methods in solving some steady state problems.
   Recently, Yang and Shu investigated superconvergence phenomenon of the DG method
   for hyperbolic conservation laws \cite{Yang;Shu:SIAM2012} and linear parabolic equations \cite{Yang;Shu:supLDG}
   in the one dimensional setting.
    Superconvergence properties of
   DG and LDG methods for hyperbolic and parabolic problems based on Fourier approach were
   studied in \cite{Guo_zhong_Qiu2013}. We also refer to
\cite{Adjerid;Weinhart2009,Adjerid;Weinhart2011,Xie;Zhang2012,Zhang;xie;zhang2009,Chen;Shu:JCP2008,Chen;Shu:SIAM2010}
  for an incomplete list of references.
   Very recently,  in \cite{Cao;zhang;zou:2k+1},  we  studied
   superconvergence properties of  a DG method for linear hyperbolic equations
   when upwind fluxes were used. We proved a
  $2k+1$th superconvergence rate of the DG approximation at the  downwind
  points (on average) as well as the domain average
  under suitable initial discretization.

  The current work is the second in a series to study superconvergence phenomena of the DG method
  in solving partial differential equations where parabolic equations are under concern.
  Our main result is a rigorous mathematical proof
   of the $2k+1$th (or $2k+1/2$th) superconvergence rate for
   the domain average and numerical fluxes at mesh nodes.
   To the best of our knowledge, the best rate proved so far in the literature is $k+2$ \cite{Yang;Shu:supLDG}.
   As a by-product, we also prove a point-wise $k+2$th superconvergence rate for the function value approximation
  and $k+1$th superconvergence rate for  the derivative approximation at the Radau (left or right)
  points.
  By doing so, we paint a full picture for superconvergence properties of the LDG method for liner parabolic equations
  in one space dimension.

   In order to establish the $2k+1$th superconvergence rate, some new analysis tools are needed.
  At the core of our analysis here is the construction of a correction function, which is super-close to the
  LDG solution. The correction function idea has been successfully applied to finite element methods (FEM) and finite volume
   methods (FVM) for elliptic equations (see, e.g. \cite{Cao;zhang;zou:2kFVM,Chen.C.M2012}),
    and more recently, to the LDG method for hyperbolic equations \cite{Cao;zhang;zou:2k+1}.
   However, the construction for parabolic equations is very different from steady state problems
   using finite element \cite{Chen.C.M2012} or finite volume methods \cite{Cao;zhang;zou:2kFVM} due to the time dependent
   effects.  Moreover, it is also quite different from the LDG method for hyperbolic equations \cite{Cao;zhang;zou:2k+1} due
   to the interplay between two correction functions.
   The main difficulty for parabolic equations lies in that
    correction functions for both variables (the exact solution $u$ and an auxiliary variable $q=u_x$)
    have to be constructed simultaneously.  To be more precise, we shall
   correct the error between the exact solution $(u,q)$ and its Gauss-Radau projection $(P_h^-u,P_h^+q)$ or $(P_h^+u,P_h^-q)$,
    depending on the choice of numerical fluxes. The construction not only is more complicated than that of
   hyperbolic equations, but also requires a novel idea to match the two variables.

  With help of the correction functions, we prove that the LDG
  solution $(u_h,q_h)$ is super-close with order $2k+1$ to our specially constructed
   interpolation function $(u_I,q_I)$ (defined in Section 3).
     It is this super-closeness that leads to the $2k+1$th superconvergence rate for the numerical fluxes at all nodes (on average)
     and for the domain average.

   To end this introduction, we would like to point out that all superconvergent results here are valid
    for one-dimensional linear systems, and the proof is along the same line without any difficulty.
     Our analysis also leads to some interesting new numerical discoveries,
    which will be reported in the last section.

  The rest of the paper is organized as follows. In Section 2, we present the LDG scheme
    for linear parabolic equations.
    Section 3 is the most technical part, where we
    construct some special functions to correct the error between the LDG
    solution and the Gauss-Radau projection of the exact solution.
   Section 4 is the main body of the paper, where superconvergence
   results are proved with suitable initial discretization.
   In Section 5, we provide some
   numerical examples to support our theoretical findings.  Finally,
    some possible future works and concluding remarks are presented in
   Section 6.

   Throughout this paper,  we adopt standard notations for Sobolev spaces such as $W^{m,p}(D)$ on sub-domain $D\subset\Omega$ equipped with
    the norm $\|\cdot\|_{m,p,D}$ and semi-norm $|\cdot|_{m,p,D}$. When $D=\Omega$, we omit the index $D$; and if $p=2$, we set
   $W^{m,p}(D)=H^m(D)$,
   $\|\cdot\|_{m,p,D}=\|\cdot\|_{m,D}$, and $|\cdot|_{m,p,D}=|\cdot|_{m,D}$. Notation``$A\lesssim B$" implies that $A$ can be
  bounded by $B$ multiplied by a constant independent of the mesh size $h$.
  ``$A\sim B$" stands for $``A\lesssim B"$ and $``B\lesssim A"$.

\section{ LDG schemes}
  We consider local discontinuous Galerkin (LDG) method for the following one-dimensional linear parabolic equation
\begin{eqnarray}\label{con_laws}
\begin{aligned}
   &u_t=u_{xx},\ \ &&(x,t)\in [0,2\pi]\times(0,T], \\
   &u(x,0)=u_0(x),\ \  &&x\in R,
\end{aligned}
\end{eqnarray}
  where $u_0$ is sufficiently smooth. We will consider both the periodic boundary
  condition $u(0,t)=u(2\pi,t)$ and the  mixed  boundary
   condition $u(0,t)=g_0(t), u_x(2\pi,t)=g_1(t)$ or  $u_x(0,t)=g_0(t),
   u(2\pi,t)=g_1(t)$.

  Let $\Omega=[0,2\pi]$ and $0=x_{\frac 12}<x_{\frac 32}<\ldots<x_{N+\frac 12} = 2\pi$
  be $N+1$ distinct points on the interval $\bar{\Omega}$.
  For  any positive integer $r$, we define $\bZ_{r}=\{1,\ldots,r\}$ and denote by
\[
    \tau_j=(x_{j-\frac 12},x_{j+\frac 12}),\ \ x_j=\frac 12(x_{j-\frac 12}+x_{j+\frac
    12}),\ j\in\bZ_N
\]
  the cells and cell centers, respectively.
  Let $h_j=x_{j+\frac 12}-x_{j-\frac 12}$,  $\bar{h}_j = h_j/2$
  and $h =  \displaystyle\max_j\; h_j$. We assume  that the mesh
  is quasi-uniform, i.e., there exists a constant $c$ such that
 $h\le c h_j$, $j\in\bZ_N$. Define the finite element space
\[
    V_h=\{ v: \; v|_{\tau_j}\in P_k(\tau_j),\; j\in\bZ_N\}^{},
\]
  where $P_k$ denotes the space of
  polynomials of degree at most $k$  with coefficients as functions of $t$.

  To construct the LDG scheme, we introduce an auxiliary variable
  $q=u_x$, then \eqref{con_laws} can be rewritten as a first order
  linear system
\begin{equation}\label{system}
   u_t=q_x, \quad q=u_x.
\end{equation}

  The LDG scheme for \eqref{con_laws} reads as: Find $u_h,q_h\in V_h$
  such that for any $v,w\in  V_h$
\begin{eqnarray}\label{LDG_scheme}
\begin{aligned}
   &(u_{ht},v)_j=-(q_{h},v_x)_j+\hat{q}_hv^-|_{j+\frac 12}-\hat{q}_hv^+|_{j-\frac 12}, \\
   &(q_h,w)_j=-(u_{h},w_x)_j+\hat{u}_hw^-|_{j+\frac 12}-\hat{u}_hw^+|_{j-\frac 12}.
\end{aligned}
\end{eqnarray}
   Here  $(u,v)_j=\int_{\tau_j}uv dx$,
   $v^-|_{j+\frac 12}$ and $v^+|_{j+\frac12}$  denote the left
  and right limits of $v$ at the point $x_{j+\frac
    12}$, respectively, and $ \hat{u}_h, \hat{q}_h$ are numerical
    fluxes. For LDG schemes, we consider alternating fluxes
\begin{equation}\label{flux1}
    \hat{u}_h=u_h^-,\ \ \ \hat{q}_h=q_h^+,
\end{equation}
  or
\begin{equation}\label{flux2}
    \hat{u}_h=u_h^+,\ \ \ \hat{q}_h=q_h^-.
\end{equation}
  In this paper, we use  both \eqref{flux1} and \eqref{flux2} as numerical
  fluxes in the periodic boundary condition, \eqref{flux1} in the
 mixed boundary condition $u(0,t)=g_0(t),
u_x(2\pi,t)=g_1(t)$, and \eqref{flux2} in the mixed  boundary
condition $u_x(0,t)=g_0(t), u(2\pi,t)=g_1(t)$.

  Define
\[
   H_h^1=\{v: \; v|_{\tau_j}\in H^1(\tau_j),\; j\in\bZ_N\}
\]
  and for all $\xi,\eta,v\in H_h^1$,  let
\[
    a^1(\xi,\eta;v)=\sum_{j=1}^{N}a^1_j(\xi,\eta;v),\ \ \ a^2(\xi,\eta;v)=\sum_{j=1}^{N}a^2_j(\xi,\eta;v)
\]
  where
\begin{eqnarray*}
 && a^1_j(\xi,\eta;v)=(\xi_t,v)_j+(\eta,v_x)_j-\hat{\eta}v^-|_{j+\frac 12}+\hat{\eta}v^+|_{j-\frac
  12},\\
  && a^2_j(\xi,\eta;v)=(\eta,v)_j+(\xi,v_x)_j-\hat{\xi}v^-|_{j+\frac 12}+\hat{\xi}v^+|_{j-\frac
  12}.
\end{eqnarray*}
  Here $\hat{\xi},\hat{\eta}$ are taken as the alternating fluxes
  \eqref{flux1} or \eqref{flux2}.
  Then the LDG scheme \eqref{LDG_scheme} can be rewritten as
\begin{equation}\label{solu1}
a^1(u_h, q_h; v)=0, \ \ \ a^2(u_h,q_h;w)=0,  \quad \forall v,w\in
V_h.
\end{equation}
Obviously, the exact solutions $u,q$  also satisfy
\begin{equation}\label{solu2}
  a^1(u, q; v)=0, \ \ \ a^2(u,q;w)=0,  \quad \forall v,w\in V_h.
\end{equation}
 By a direct calculation, there hold
\begin{equation}\label{p1}
   a^1(v, w; v)+a^2(v, w; w)=(v_t,v)+(w,w)-w^+v^-|_{N+\frac 12}+w^+v^-|_{\frac 12}
\end{equation}
 for the fluxes choice \eqref{flux1} and
\begin{equation}\label{p2}
    a^1(v, w; v)+a^2(v, w; w)=(v_t,v)+(w,w)-w^-v^+|_{N+\frac 12}+w^-v^+|_{\frac 12}.
\end{equation}
for the fluxes choice \eqref{flux2}.


\section{ Construction of special interpolation functions}
\setcounter{equation}{0}
 Our goal here is
 to construct  a special interpolation function $(u_I,q_I)$, which is
  superclose to the LDG solution $(u_h,q_h)$.

  We begin with some preliminaries. First, for any $r$, we denote by $\lfloor r\rfloor$ the maximal integer
  no more than $r$, and $\lceil r\rceil$ the minimal integer  no less than $r$. Next, we define
   on $v\in H_h^1$, two Gauss-Radau projections $P_h^-,P_h^+$ by
\begin{eqnarray*}\label{gr}
&(P^-_hv,w)_j=(v,w)_j,\forall
w\in\bP^{k-1}(\tau_j)\quad\text{and}\quad
P^-_hv(x_{j+\frac12}^-)=v(x_{j+\frac12}^-),\\
\label{gr1}
 &(P^+_hv,w)_j=(v,w)_j,\forall
w\in\bP^{k-1}(\tau_j)\quad\text{and}\quad
P^+_hv(x_{j-\frac12}^+)=v(x_{j-\frac12}^+),
\end{eqnarray*}
   and  an integral operator $D^{-1}_s$ by
\[
     D^{-1}_sv(x)=\frac{1}{\bar{h}_j}\int_{x_{j-\frac 12}}^xv(x') dx'=\int_{-1}^s\hat{v}(s')ds', \quad x\in
     \tau_j, j\in\bZ_N,
\]
  where
\[
    s=(x-x_j)/\bar{h}_j\in[-1,1],\ \ \hat{v}(s)=v(x).
\]

   We have, for any function $v\in H_h^1$, the following Legendre expansion in each element $\tau_j, j\in\bZ_N$,
\[
   v(x,t)=\sum_{m=0}^{\infty}v_{j,m}(t)L_{j,m}(x),\ \ \  v_{j,m}=\frac{2m+1}{h_j}(v,L_{j,m})_j,
\]
  where $L_{j,m}$ denotes the  normalized Legendre polynomial of degree $m$ on
  $\tau_j$. By the definition of $P_h^-,P_h^+$,
\begin{eqnarray*}\label{expansion}
   &&(v-P^-_hv)(x,t)=\bar{v}_{j,k}(t)L_{j,k}+\sum_{m=k+1}^\infty
   v_{j,m}(t)L_{j,m}(x),\\
   &&(v-P^+_hv)(x,t)=\tilde{v}_{j,k}(t)L_{j,k}+\sum_{m=k+1}^\infty
   v_{j,m}(t)L_{j,m}(x),
\end{eqnarray*}
  where
\begin{eqnarray}\label{coefficient}
&&\bar{v}_{j,k}=-v(x^-_{j+\frac 12},t)+\frac{1}{h_j}\int_{\tau_j}
v(x,t)\sum_{m=0}^{k}(2m+1)L_{j,m}(x)dx, \\
\label{coefficient_0}
 &&\tilde{v}_{j,k}=(-1)^{k+1}v(x^+_{j-\frac
12},t)+\frac{1}{h_j}\int_{\tau_j}
v(x,t)\sum_{m=0}^{k}(-1)^{k+m}(2m+1)L_{j,m}(x)dx.
\end{eqnarray}
  Obviously,
\begin{equation}\label{inter_rep}
   (v-P^-_hv, w)_j=\bar{v}_{j,k}(L_{j,k},w)_j,\ \ \ (v-P^+_hv, w)=\tilde{v}_{j,k}(L_{j,k},w)_j,\ \ \forall w\in V_h.
\end{equation}

   In each element $\tau_j, j\in\bZ_N$, we define
\begin{eqnarray}\label{correcti}
  && F_{1,1}=P_h^+D^{-1}_sL_{j,k},\ \ F_{1,i}=(P_h^+D^{-1}_sP_h^-D^{-1}_s)^iF_{1,1},\ i\ge 2,\\
  \label{correctii}
  && F_{2,1}=P_h^-D^{-1}_sL_{j,k},\ \ F_{2,i}=(P_h^-D^{-1}_sP_h^+D^{-1}_s)^{i}F_{2,1},\ i\ge 2.
\end{eqnarray}


\begin{lemma}\label{lemma:1}
  For all $1\le i\le  \lceil k/2\rceil, x\in\tau_j,j\in\bZ_N$,
  $F_{1,i}(x), F_{2,i}(x)$ have the following representations
\begin{eqnarray}\label{fir}
 &F_{1,i}(x)=\sum_{m=k-2i+2}^ka_{i,m}(L_{j,m}+L_{j,m-1})(x),\\
 \label{fir2}
 &F_{2,i}(x)=\sum_{m=k-2i+2}^kb_{i,m}(L_{j,m}-L_{j,m-1})(x),
\end{eqnarray}
  where the coefficients $a_{i,m},b_{i,m}$ are some bounded constants independent of the mesh
size $h_j$. Consequently,
\begin{eqnarray}\label{fp1}
& F_{1,i}(x_{j-\frac12}^+)=0, \quad \|F_{1,i}\|_{0,\infty,\tau_j}\lesssim 1, \\
\label{fp2}
 &F_{2,i}(x_{j+\frac12}^-)=0, \quad \|F_{2,i}\|_{0,\infty,\tau_j}\lesssim 1.
\end{eqnarray}
\end{lemma}
\begin{proof}
  For all $m\ge 1$, noticing that $\|L_{j,m}\|_{0,\infty,\tau_j}=1$ and
\[
  (L_{j,m}+L_{j,m-1})(x^+_{j-\frac 12})=0,\ \  (L_{j,m}-L_{j,m-1})(x^-_{j+\frac 12})=0,
\]
  then \eqref{fp1}-\eqref{fp2} follow directly from \eqref{fir}-\eqref{fir2}.

   In the following, We shall focus our attention on \eqref{fir} since \eqref{fir2} can be obtained by following the
   same line. We show \eqref{fir} by induction. First, by the definition of $P_h^+$ and the fact that
\begin{equation}\label{ll}
 D^{-1}_sL_{j,m}=\frac{1}{2m+1}(L_{j,m+1}-L_{j,m-1}),\  m\ge 1,
 \end{equation}
  we derive
\begin{equation*}\label{F1z}
   F_{1,1}=-\frac{1}{2k+1}(L_{j,k}+L_{j,k-1}),
\end{equation*}
which implies \eqref{fir} is valid for $i=1$ with
$a_{1,k}=-\frac{1}{2k+1}.$ Now we suppose \eqref{fir} is valid for
$i, i\le \lceil k/2\rceil-1$. Since
\[
 P_h^-L_{j,k+1}=L_{j,k}, \quad P_h^+L_{j,k+1}=-L_{j,k},   \quad P_hL_{j,m}=L_{j,m}, \quad 1\le m\le k,
 \]
 where $P_h=P_h^-$ or $P_h^+$, it is easy to deduce from \eqref{ll} that
\begin{eqnarray*}
    &&P_h^-D^{-1}_sL_{j,k}=\frac{1}{2k+1}(L_{j,k}-L_{j,k-1}),\ \ P_h^+D^{-1}_sL_{j,k}=\frac{-1}{2k+1}(L_{j,k}+L_{j,k-1}),\\
    && P_hD^{-1}_sL_{j,m}=\frac{1}{2m+1}(L_{j,m+1}-L_{j,m-1}),\ \  1\le m\le
    {k-1}.
\end{eqnarray*}
 Therefore,
\begin{equation}\label{P-FF}
  P_h^{-}D_s^{-1}F_{1,i}=\sum_{m=k-2i+1}^{k}\beta_{i,m}(L_{j,m}-L_{j,m-1}),
\end{equation}
  where
\[
   \beta_{i,m} =\frac{a_{i,m+1}+a_{i,m}}{2m+1}+\frac{a_{i,m}+a_{i,m-1}}{2m-1}
\]
  with $a_{i,k+1}=a_{i,k-2i+1}=a_{i,k-2i}=0$. Now we consider $F_{1,i+1}$. Note that
\[
   F_{1,i+1}=P_h^{+}D_s^{-1}P_h^{-}D_s^{-1}F_{1,i},
\]
  we have from \eqref{P-FF}
\begin{eqnarray*}
   F_{1,i+1}&=&\sum_{m=k-2i+1}^{k}\beta_{i,m}P_h^{+}D_s^{-1}(L_{j,m}-L_{j,m-1})\\
   &=&
   \sum_{m=k-2i}^{k}a_{i+1,m}(L_{j,m}+L_{j,m-1}),
\end{eqnarray*}
 where
\[
   a_{i+1,m}=\frac{\beta_{i,m+1}-\beta_{i,m}}{2m+1}+\frac{\beta_{i,m-1}-\beta_{i,m}}{2m-1}
\]
  with $\beta_{i,k+1}=\beta_{i,k-2i}=\beta_{i,k-2i-1}=0$.
Consequently, \eqref{fir} is valid for $i+1$. Then \eqref{fir}
follows. This completes our proof.
\end{proof}

    With the functions  $F_{1,i},F_{2,i}$, we define in each $\tau_j,j\in\bZ_N$ other two functions
    $\bar{F}_{1,i},\bar{F}_{2,i}$ as
\begin{equation}\label{barF}
   \bar{F}_{1,i}=P_h^{-}D_s^{-1}F_{1,i},\ \
   \bar{F}_{2,i}=P_h^{+}D_s^{-1}F_{2,i},\  1\le i \le \lfloor k/2 \rfloor.
\end{equation}
  By the same arguments as in Lemma \ref{lemma:1}, we obtain
\begin{eqnarray}\label{P-F1}
  &&\bar{F}_{1,i}=\sum_{m=k-2i+1}^{k}\beta_{i,m}(L_{j,m}-L_{j,m-1}),\\
  \label{P-F2}
  &&\bar{F}_{2,i}=\sum_{m=k-2i+1}^{k}\gamma_{i,m}(L_{j,m}+L_{j,m-1}),
\end{eqnarray}
  where  $\beta_{i,m}, \gamma_{i,m}$  are constants independent of $h_j$.
  Consequently,
\begin{eqnarray}\label{fp3}
& \bar{F}_{1,i}(x_{j+\frac12}^-)=0, \quad \|\bar{F}_{1,i}\|_{0,\infty,\tau_j}\lesssim 1, \\
\label{fp4}
 &\bar{F}_{2,i}(x_{j-\frac12}^+)=0, \quad \|\bar{F}_{2,i}\|_{0,\infty,\tau_j}\lesssim 1.
\end{eqnarray}
   In addition,   a straightforward calculation from \eqref{correcti}-\eqref{correctii} and
   \eqref{barF} yields
\begin{equation}\label{cf1}
   F_{1,i+1}=P^+_hD_s^{-1}\bar{F}_{1,i},\ \ \quad
   F_{2,i+1}=P^-_hD_s^{-1}\bar{F}_{2,i},\  1\le i \le \lfloor k/2 \rfloor.
\end{equation}

\subsection{Correction functions for the fluxes \eqref{flux1}}
  In each element $\tau_j,j\in\bZ_N$, we have, from \eqref{inter_rep},
\begin{equation}\label{iter_rep1}
   (u-P_h^-u,v)_j=\bar{u}_{j,k}(t)(L_{j,k},v)_j,\ \  (q-P_h^+q,v)_j=\tilde{q}_{j,k}(t)(L_{j,k},v)_j,\ \ \forall v\in V_h,
\end{equation}
  where the coefficients $\bar{u}_{j,k},\tilde{q}_{j,k}$ are given by \eqref{coefficient}-\eqref{coefficient_0}.
  Let
\[
G_i(t)=\bar{u}^{(i)}_{j,k}(t),\ \ \
   Q_i(t)=\tilde{q}^{(i)}_{j,k}(t),\ \ 0\le i\le \lceil
   k/2\rceil.
\]
By the standard approximation theory, if $u\in
W^{k+2+2i,\infty}(\Omega)$,
\begin{eqnarray}\label{appro}
   &&|G_{i}|=|D_t^{i}\bar{u}_{j,k}|\lesssim h^{k+1}\|\partial^i_tu\|_{k+1,\infty,\tau_j}\lesssim h^{k+1}\|u\|_{k+1+2i,\infty,\tau_j}, \\
   \label{appro1}
   && |Q_{i}|=|D_t^{i}\tilde{q}_{j,k}|\lesssim
   h^{k+1}\|\partial^i_tu\|_{k+2,\infty,\tau_j}\lesssim
   h^{k+1}\|u\|_{k+2+2i,\infty,\tau_j}.
\end{eqnarray}

  Now we are ready to construct our correction functions. For all $1\le l\le
  k$,
we define, first at the boundary points $x=x_{ \frac 12}$ and $x=x_{
N+\frac 12}$,
\begin{equation}\label{corr_boundary}
   W^l_1(x^+_{ N+\frac 12},t)=0,\ \ \  W^l_2(x^-_{ \frac 12},t)=0,\ \ \forall t\ge 0,
\end{equation}
  and then in each element $\tau_j, j\in\bZ_N$,
\begin{equation}\label{corr_func}
   W^l_1(x,t)=\sum_{i=1}^{\lceil l/2\rceil}{w}_{1,i}+\sum_{i=1}^{\lfloor l/2\rfloor}\bar{w}_{2,i},\ \
   W_2^l(x,t)=\sum_{i=1}^{\lfloor l/2\rfloor}\bar{w}_{1,i}+\sum_{i=1}^{\lceil l/2\rceil}{w}_{2,i},
\end{equation}
   where
\begin{eqnarray}\label{correctu}
    && {w}_{1,i}=\bar{h}_j^{2i-1}G_i{F}_{1,i},\ \ \  \ \bar{w}_{1,i}=\bar{h}_j^{2i}G_i\bar{F}_{1,i},\\
    \label{correctq}
   && {w}_{2,i}=\bar{h}_j^{2i-1}Q_{i-1}{F}_{2,i}, \ \ \  \bar{w}_{2,i}=\bar{h}_j^{2i}Q_i\bar{F}_{2,i}.
\end{eqnarray}

\begin{lemma}\label{lemma2}
 Suppose $W^l_1,W^l_2\in V_h$ are defined by
\eqref{corr_boundary}-\eqref{correctq}. Then
\begin{eqnarray}\label{corr1}
    W^l_1(x_{j-\frac 12}^+,t)=0,\ \ \ W^l_2(x_{j-\frac 12}^-,t)=0,\ \
    \forall j\in\bZ_{N+1}.
\end{eqnarray}
Moreover, if $l=2r$ is even,
\begin{eqnarray}\label{cor1}
    &(W^l_{2t},v)_j+(W^l_1,v_x)_j=(w_{1,1},v_x)_j+(\bar{w}_{1,rt},v)_j\\
    \label{cor2}
    &(W^l_2,v_x)_j+(W^l_1,v)_j=(w_{2,1},v_x)_j+(\bar{w}_{2,r},v)_j,
\end{eqnarray}
   if $l=2r+1$ is odd,
\begin{eqnarray}\label{cor3}
    &(W^l_{2t},v)_j+(W^l_1,v_x)_j=(w_{1,1},v_x)_j+(w_{2,r+1t},v)_j\\
    \label{cor4}
    &(W^l_2,v_x)_j+(W^l_1,v)_j=(w_{2,1},v_x)_j+(w_{1,r+1},v)_j.
\end{eqnarray}

\begin{proof}
   By \eqref{fp1}-\eqref{fp2} and \eqref{fp3}-\eqref{fp4},
\[
   w_{1,i}(x_{j-\frac 12}^+,t)=\bar{w}_{2,i}(x_{j-\frac 12}^+,t)=0,\ \
   w_{2,i}(x_{j+\frac 12}^-,t)=\bar{w}_{1,i}(x_{j+\frac 12}^-,t)=0,\
   \ j\in\bZ_N.
\]
    Then \eqref{corr1} follows from \eqref{corr_boundary}-\eqref{corr_func}.

   We now show \eqref{cor1}-\eqref{cor4}. For any integer $l, 1\le l\le k$,
   a direct calculation from \eqref{fir}-\eqref{fir2} and
   \eqref{P-F1}-\eqref{P-F2} gives
\[
   D_s^{-1}{F}_{1,i}(x^-_{j+\frac 12})=D_s^{-1}{F}_{1,i}(x^+_{j-\frac
   12})=0,\ \ D_s^{-1}{F}_{2,i}(x^-_{j+\frac 12})=D_s^{-1}{F}_{2,i}(x^+_{j-\frac 12})=0
\]
  for all $i\in\bZ_{\lfloor l/2\rfloor}$, and
\[
   D_s^{-1}\bar{F}_{1,i}(x^-_{j+\frac
12})=D_s^{-1}\bar{F}_{1,i}(x^+_{j-\frac 12})=0,\ \
D_s^{-1}\bar{F}_{2,i}(x^-_{j+\frac
12})=D_s^{-1}\bar{F}_{2,i}(x^+_{j-\frac 12})=0
\]
  for all $i\in\bZ_{\lfloor l/2\rfloor-1}$ in case $l=2r$ and  $i\in\bZ_{\lfloor
  l/2\rfloor}$ in case $l=2r+1$. Then by integration by parts, \eqref{barF} and
  \eqref{cf1},
\begin{eqnarray*}
   (\bar{w}_{1,i},v_x)_j+({w}_{1,i},v)_j&=&\bar{h}_j^{2i}G_i(\bar{F}_{1,i},v_x)_j+\bar{h}_j^{2i-1}G_i(F_{1,i},v)_j\\
   &=&\bar{h}_j^{2i}G_i(\bar{F}_{1,i}-D_s^{-1}F_{1,i},v_x)_j=0,\\
  (w_{2,it},v)_j+(\bar{w}_{2,i},v_x)_j&=&\bar{h}_j^{2i-1}Q_{i}({F}_{2,i},v)_j+\bar{h}_j^{2i}Q_{i}(\bar{F}_{2,i},v_x)_j\\
  &=&\bar{h}_j^{2i}Q_{i}(\bar{F}_{2,i}-D_s^{-1}F_{2,i},v_x)_j=0
\end{eqnarray*}
  for all $i\in\bZ_{\lfloor l/2\rfloor}$, and
\begin{eqnarray*}
   (\bar{w}_{1,it},v)_j+({w}_{1,i+1},v_x)_j&=&\bar{h}_j^{2i}G_{i+1}(\bar F_{1,i},v)_j+\bar{h}_j^{2i+1}G_{i+1}({F}_{1,i+1},v_x)_j\\
    &=&\bar{h}_j^{2i+1}G_{i+1}({F}_{1,i+1}-D_s^{-1}\bar{F}_{1,i},v_x)_j=0,\\
 (w_{2,i+1},v_x)_j+(\bar{w}_{2,i},v)_j&=&\bar{h}_j^{2i+1}Q_{i}({F}_{2,i+1},v_x)_j+\bar{h}_j^{2i}Q_{i}(\bar F_{2,i},v)_j\\
 &=&\bar{h}_j^{2i+1}Q_{i}(F_{2,i+1}-D_s^{-1}\bar{F}_{2,i},v_x)_j=0
\end{eqnarray*}
  for all $i\in\bZ_{\lfloor l/2\rfloor-1}$ in case $l=2r$ and  $i\in\bZ_{\lfloor
  l/2\rfloor}$ in case $l=2r+1$.
Then the desired results \eqref{cor1}-\eqref{cor4} follow by summing
over all $i$.
\end{proof}
\end{lemma}

  With the correction functions $W^l_1,W^l_2, 1\le l\le k$, we define the special interpolation functions
\begin{equation}\label{interpolation}
  u^l_I=P_h^-u-W^l_2,\ \ \ q^l_I=P_h^+q-W^l_1.
\end{equation}
  By \eqref{corr1}, we have
\begin{equation}\label{interp:1}
   u^l_I(x^-_{j-\frac 12},t)=u(x^-_{j-\frac 12},t),\ \
\ q^l_I(x^+_{j-\frac 12},t)=q(x^+_{j-\frac
   12},t),\ \ \forall j\in\bZ_{N+1}.
\end{equation}

\subsection{Correction functions for the fluxes \eqref{flux2}}
  In this case, we still use the notation
\[
G_i(t)=\tilde{u}^{(i)}_{j,k}(t),\ \ \
   Q_i(t)=\bar{q}^{(i)}_{j,k}(t),\ \ 0\le i\le \lceil k/2\rceil,
\]
where $\tilde{u}_{j,k},\bar{q}_{j,k}$ are defined by
\eqref{coefficient}-\eqref{coefficient_0}.

 Similar as the fluxes choice \eqref{flux1}, we  construct  the correction functions as follows.
 For all $1<l\le k, t\ge 0$, we define, at the boundary points $x=x_{ \frac 12}$ and $x=x_{
N+\frac 12}$,
\begin{equation}\label{corr_boundary1}
   W^l_1(x^-_{ \frac 12},t)=0,\ \ \  W^l_2(x^+_{ N+\frac 12},t)=0,
\end{equation}
  and in each element $\tau_j, j\in\bZ_N$,
\begin{equation}\label{f1}
  W^l_1(x,t)=\sum_{i=1}^{\lfloor l/2\rfloor}\bar{w}_{1,i}+\sum_{i=1}^{\lceil
  l/2\rceil}{w}_{2,i},\ \
   W_2^l(x,t)=\sum_{i=1}^{\lceil l/2\rceil}{w}_{1,i}+\sum_{i=1}^{\lfloor l/2\rfloor}\bar{w}_{2,i},
\end{equation}
   where
\begin{eqnarray*}\label{f2}
    && {w}_{1,i}=\bar{h}_j^{2i-1}Q_{i-1}{F}_{1,i},\ \ \ \bar{w}_{1,i}=\bar{h}_j^{2i}Q_i\bar{F}_{1,i},\\
    \label{f3}
   && {w}_{2,i}=\bar{h}_j^{2i-1}G_i{F}_{2,i}, \ \ \  \ \bar{w}_{2,i}=\bar{h}_j^{2i}G_i\bar{F}_{2,i}.
\end{eqnarray*}
 We define the special interpolation functions in each element $\tau_j, j\in\bZ_N$  as
\begin{equation}\label{interpolation1}
  u^l_I=P_h^+u-W^l_2,\ \ \ q^l_I=P_h^-q-W^l_1,\ \ 1\le l\le k.
\end{equation}
  A direct calculation yields
\begin{equation}\label{interp:2}
   u^l_I(x^+_{j-\frac 12},t)=u(x^+_{j-\frac 12},t),\ \ \ q^l_I(x^-_{j-\frac 12},t)=q(x^-_{j-\frac
   12},t),\ \ \forall j\in\bZ_{N+1}.
\end{equation}

   We end with this section some estimates for $W_1^l,W^l_2$ and the
   interpolation function
   $(u_I,q_I)$, which play important roles in our superconvergence
   analysis.
\begin{theorem}\label{theorem1}
  Let $u\in W^{k+l+2,\infty}(\Omega), 1\le l\le k$ be the solution of
  \eqref{con_laws}. Suppose  $W^l_1,W^l_2$  are defined by
  \eqref{corr_boundary}-\eqref{corr_func} for fluxes \eqref{flux1} or \eqref{corr_boundary1}-\eqref{f1} for fluxes \eqref{flux2}.
  Then for all $j\in\bZ_N$
\begin{equation}\label{estimate:w}
    \|W^l_1\|_{0,\infty,\tau_j} + \|W^l_2\|_{0,\infty,\tau_j}\lesssim h^{k+2}\|u\|_{k+l+2,\infty,\tau_j}.
\end{equation}
  Moreover, if $(u^l_I,q^l_I)\in V_h$ is the corresponding interpolation function defined by
  \eqref{interpolation}  and \eqref{interpolation1} for fluxes
  \eqref{flux1} and  \eqref{flux2}, respectively,
\begin{eqnarray}\label{corr2}
   && \left|(u^l_{It}-u_t,v)_j-(W^l_1,v_x)_j\right|\lesssim h^{k+l+1}\|u\|_{k+2+l,\infty,\tau_j}\|v\|_{0,1,\tau_j},\\
    \label{corr3}
    && \left|(q^l_I-q,v)_j-(W^l_2,v_x)_j\right|\lesssim h^{k+l+1}\|u\|_{k+2+l,\infty,\tau_j}\|v\|_{0,1,\tau_j}.
\end{eqnarray}
\end{theorem}
\begin{proof}   We only consider the
  fluxes \eqref{flux1}, since the the proof for fluxes \eqref{flux2}
  is following the same line.
For all $i\ge 1$, as direct consequences of the second inequality of
\eqref{fp1}-\eqref{fp2} and
 \eqref{fp3}-\eqref{fp4}, and \eqref{appro}-\eqref{appro1},
 \begin{eqnarray}\label{e:1}
   &&\|w_{1,i}\|_{0,\infty,\tau_j}\lesssim h^{k+2i}\|u\|_{k+1+2i},\
   \ \|w_{2,i}\|_{0,\infty,\tau_j}\lesssim h^{k+2i}\|u\|_{k+2i}, \\
   \label{e:2}
 &&\|\bar w_{1,i}\|_{0,\infty,\tau_j}\lesssim
   h^{k+2i+1}\|u\|_{k+1+2i}
   \ \ \|\bar w_{2,i}\|_{0,\infty,\tau_j}\lesssim
   h^{k+2i+1}\|u\|_{ k+2+2i}.
  \end{eqnarray}
    Then \eqref{estimate:w} follows.

  We now show \eqref{corr2}-\eqref{corr3}.
  By \eqref{iter_rep1}, integration by parts, and the first formula of
\eqref{correcti}-\eqref{correctii},
\begin{eqnarray*}
    &&(P^-_hu_t-u_t,v)_j=-G_1(L_{j,k},v)_j=\bar{h}_jG_1(F_{1,1},v_x)_j=(w_{1,1},v_x)_j,\\
  &&(P^+_hq-q,v)_j=-Q_0(L_{j,k},v)_j=\bar{h}_jQ_0(F_{2,1},v_x)_j=(w_{2,1},v_x)_j.
\end{eqnarray*}
   Then
\begin{eqnarray*}
   &&(u^l_{It}-u_t,v)-(W_1,v_x)=(w_{1,1},v_x)_j-(W^l_{2t},v)-(W^l_1,v_x),\\
   &&(q^l_I-q,v)-(W_2,v_x)=(w_{2,1},v_x)_j-(W^l_{1},v)-(W^l_2,v_x).
\end{eqnarray*}
   In light of \eqref{cor1}-\eqref{cor4}, we have
\[
   (u^l_{It}-u_t,v)_j-(W^l_1,v_x)_j=(\bar{w}_{1,rt},v)_j,\ \ (q^l_I-q,v)_j-(W^l_2,v_x)_j=(\bar{w}_{2,r},v)_j
\]
  for $l=2r$ and
\[
   (u^l_{It}-u_t,v)_j-(W^l_1,v_x)_j=(w_{2,r+1t},v)_j,\ \ (q^l_I-q,v)_j-(W^l_2,v_x)_j=(w_{1,r+1},v)_j
\]
  for $l=2r+1$. By \eqref{e:1}-\eqref{e:2},  we have for all $l\ge 1$,
\begin{eqnarray*}
  &&|(u^l_{It}-u_t,v)_j-(W^l_1,v_x)_j|\lesssim h^{k+l+1} \|u\|_{k+l+2,\infty,\tau_j}\|v\|_{0,1,\tau_j},\\
  &&|(q^l_I-q,v)_j-(W^l_2,v_x)_j|\lesssim h^{k+l+1} \|u\|_{k+l+2,\infty,\tau_j}\|v\|_{0,1,\tau_j}.
\end{eqnarray*}
  The proof is completed.
\end{proof}

\section{ Superconvergence}

   In this section, we shall study superconvergence properties of the LDG
   solution at some special points : nodes, left and right Radau points,
   and superconvergence for the domain and cell average.  We denote by $R^l_{j,m},R^r_{j,m}, m\in \bZ_{k}$ the
   $k$ interior left and right Radau points in the interval $\tau_j,
   j\in\bZ_N$, respectively.  Namely,
   $R_{j,m}^l, m\in\bZ_k$ are zeros of $L_{j,k+1}+L_{j,k}, j\in\bZ_N$ except the point $x=x_{j-\frac
  12}$, and  $R_{j,m}^r, m\in\bZ_k$  are zeros of $L_{j,k+1}-L_{j,k}$ except the point $x=x_{j+\frac
12}$.

  We begin with a study of the error between the LDG solution
   $(u_h,q_h)$ and the interpolation
  function  $(u^l_I,q^l_I), 1\le l\le k$ defined in  \eqref{interpolation} or
  \eqref{interpolation1}.

\begin{theorem}\label{theo:0}
    Let $u\in W^{k+l+2,\infty}(\Omega),1\le l\le k $  be the solution of
   \eqref{con_laws}, and  $u_h,q_h\in V_h$ the solution of \eqref{LDG_scheme}.
   Let $u^l_I,q^l_I\in V_h$  be defined by \eqref{interpolation}
   for fluxes \eqref{flux1} or \eqref{interpolation1} for fluxes
   \eqref{flux2}. Suppose the initial solution
   $u_h(\cdot,0)=u_I^l(\cdot,0)$.
   Then for both the periodic and mixed  boundary
  conditions,
\begin{eqnarray}\label{spclossness}
   \begin{aligned}
 && \|u^l_I-u_h\|_0(t)
   &\lesssim (1+t)h^{k+l+1}\|u\|_{k+l+2,\infty},\\
  && \|q^l_I-q_h\|_0(t)
   &\lesssim (1+t^{\frac 12})h^{k+l+1}\|u\|_{k+l+2,\infty}.
 \end{aligned}
 \end{eqnarray}
\end{theorem}
\begin{proof} Let $\eta_u=u^l_I-u_h, \eta_q=q^l_I-q_h$. Recall the definition
of $a^1(\cdot,\cdot;\cdot),a^2(\cdot,\cdot;\cdot)$ and
\eqref{solu1}-\eqref{solu2}, we have for all $v,w\in V_h$,
\begin{eqnarray*}
   &&a^1(\eta_u,\eta_q;v)=(u^l_{It}-u_t,v)-(W^l_1,v_x),\\
  &&
  a^2(\eta_u,\eta_q;w)=(q^l_I-q,w)-(W^l_2,w_x).
\end{eqnarray*}
  By Theorem \ref{theorem1},
  the inequalities \eqref{corr2}-\eqref{corr3} hold for both the fluxes \eqref{flux1} and
 \eqref{flux2}, then
\begin{eqnarray*}
   &&|a^1(\eta_u,\eta_q;v)|\lesssim h^{k+l+1}\|u\|_{k+l+2,\infty}\|v\|_{0,1},\\
  &&
  |a^2(\eta_u,\eta_q;w)|\lesssim h^{k+l+1}\|u\|_{k+l+2,\infty}\|w\|_{0,1}.
\end{eqnarray*}
   We now show \eqref{spclossness}. We first consider
   the periodic boundary condition.
   Since
\begin{eqnarray*}
   &&(u^l_I-u_h)^-_{N+\frac 12}=(u^l_I-u_h)^-_{\frac 12},\ \ \  (u^l_I-u_h)^+_{N+\frac 12}=(u^l_I-u_h)^+_{\frac
   12},\\
   &&(q^l_I-q_h)^+_{N+\frac 12}=(q^l_I-q_h)^+_{\frac 12},\ \ \ (q^l_I-q_h)^-_{N+\frac 12}=(q^l_I-q_h)^-_{\frac
   12},
\end{eqnarray*}
 by choosing  $v=\eta_u,
   w=\eta_q$ in \eqref{p1} for fluxes \eqref{flux1}, or in \eqref{p2} for fluxes\eqref{flux2}, we obtain
   for both fluxes choice
\begin{eqnarray*}
  (\eta_{ut},\eta_u)+(\eta_q,\eta_q)&=& a^1(\eta_u,\eta_q;\eta_u)+a^2(\eta_u,\eta_q;\eta_q)\\
          &\lesssim
          &h^{k+l+1}\|u\|_{k+l+2,\infty}(\|\eta_u\|_{0,1}+\|\eta_q\|_{0,1}).
\end{eqnarray*}
   By Cauchy-Schwarz inequality, we get
\begin{equation}\label{iii1}
  \frac{1}{2} \frac {d}{dt}\|\eta_u\|^2_0=(\eta_{ut},\eta_u)\lesssim
  h^{k+l+1}\|u\|_{k+l+2,\infty}(\|\eta_u\|_{0}+h^{k+l+1}\|u\|_{k+l+2,\infty}).
\end{equation}
  Due to the special choice of initial condition, we have $\|\eta_u\|_0(0)=0$, which
  yields
\[
   \|\eta_u\|^2_0(t)=\int_{0}^t\frac{d}{dt}\|\eta_u\|^2_0 dt\lesssim
   th^{k+l+1}\|u\|_{k+l+2,\infty}(\|\eta_u\|_{0}(t)+h^{k+l+1}\|u\|_{k+l+2,\infty}).
\]
   Then the first inequality of \eqref{spclossness} follows from a
   direct calculation.
 Note that
\begin{eqnarray*}
  \|\eta_q\|_0^2 \lesssim
 h^{k+l+1}\|u\|_{k+l+2,\infty}\|\eta_u\|_0+h^{k+l+1}\|u\|_{k+l+2,\infty}\|\eta_q\|_0,
\end{eqnarray*}
  we obtain
\[
   \|\eta_q\|_0\lesssim (1+t^{\frac 12})h^{k+l+1}\|u\|_{k+l+2,\infty}.
\]
    This finishes the second inequality of \eqref{spclossness} for the periodic boundary
   condition.

   Now we consider the mixed  boundary condition. Noticing that
\[
   (u^l_I-u_h)^-_{\frac 12}=0,\ \ \ (q^l_I-q_h)^+_{N+\frac 12}=0
\]
  for the condition $u(0,t)=g_0(t), u_x(2\pi,t)=g_1(t)$ and
\[
   (q^l_I-q_h)^-_{\frac 12}=0,\ \ \ (u^l_I-u_h)^+_{N+\frac 12}=0
\]
   for the condition $u_x(0,t)=g_0(t), u(2\pi,t)=g_1(t)$,  by choosing  $v=\eta_u,
   w=\eta_q$ in \eqref{p1} and \eqref{p2}, respectively, we derive in
   both cases
\[
   (\eta_{ut},\eta_u)+(\eta_q,\eta_q)= a^1(\eta_u,\eta_q;\eta_u)+a^2(\eta_u,\eta_q;\eta_q).
\]
  Following the same line as in the periodic case, we obtain
  \eqref{spclossness} directly for the mixed  boundary condition.
\end{proof}

\begin{remark}
   By choosing $l=k$ in Theorem \ref{theo:0},  the special interpolation
   function $(u^k_I,q^k_I)$ is superclose to the LDG solution
   $(u_h,q_h)$, with a superconvergence rate $2k+1$. It is the
   supercloseness that leads to the $2k+1$ superconvergence rate at
   nodes as well as the domain average.
\end{remark}

As direct consequences of \eqref{spclossness} and the estimates for
the correction functions $W^l_1,W^l_2$ in \eqref{estimate:w}, we
have the following superconvergence results for the Gauss-Radau
projections of the exact solution.

\begin{corollary}\label{coro:1}
 Let $u\in W^{k+4,\infty}(\Omega)$  be the solution of
   \eqref{con_laws} and  $u_h,q_h\in V_h$ the solution of \eqref{LDG_scheme}, respectively.
     Suppose the initial solution
   $u_h(\cdot,0)=u_I^l(\cdot,0), l=2$ with $u_I^l$ defined by \eqref{interpolation}
   for fluxes \eqref{flux1}, or \eqref{interpolation1} for fluxes
   \eqref{flux2}.
   Then for both the periodic and mixed  boundary
  conditions,
\begin{equation}\label{xi:1}
   \|\xi_u\|_0\lesssim (1+th) h^{k+2}\|u\|_{k+4,\infty},\ \ \|\xi_q\|_0\lesssim
   (1+t^{\frac 12}h)h^{k+2}\|u\|_{k+4,\infty},
\end{equation}
\end{corollary}
  where $\xi_u=P_h^-u-u_h, \xi_q=P_h^+q-q_h$ for fluxes
  \eqref{flux1} and $\xi_u=P_h^+u-u_h, \xi_q=P_h^-q-q_h$ for fluxes
  \eqref{flux2}.

\subsection{Superconvergence of the numerical fluxes at nodal points}

We are now ready to present our superconvergence results of the
numerical fluxes at nodes.
\begin{theorem}\label{theo:1}
    Let $u\in W^{2k+2,\infty}(\Omega)$  be the solution of
   \eqref{con_laws}, and $u_h, q_h$  the solution of
   \eqref{LDG_scheme}.
    Suppose the initial solution  $u_h(\cdot,0)=u^k_I(\cdot,0)$ with $u^k_I(\cdot,0)$ defined
    by \eqref{interpolation} for fluxes \eqref{flux1}, or  \eqref{interpolation1} for fluxes \eqref{flux2}. Then
    for both the periodic and mixed  boundary conditions,
\begin{eqnarray}\label{super_node2}
  && e_{u,n}\lesssim (1+t)h^{2k+\frac
   12}\|u\|_{2k+2,\infty},\ \ e_{q,n}\lesssim (1+ t^{\frac 12})h^{2k+\frac
   12}\|u\|_{2k+2,\infty},\\
  \label{super_node3}
  &&\|e_u\|_*\lesssim (1+t)h^{2k+1}\|u\|_{2k+2,\infty},\ \ \|e_q\|_*\lesssim
  (1+ t^{\frac 12})h^{2k+1}\|u\|_{2k+2,\infty},
\end{eqnarray}
  where
\begin{eqnarray*}
   &&e_{u,n}=\max_{j\in\bZ_{N+1}}\left|(u-\hat{u}_h)(x_{j-\frac 12},t)\right|,
\;\; \|e_u\|_{*}=\left(\frac
1{N+1}\sum_{j=1}^{N+1}\big(u-\hat{u}_h\big)^2\big(x_{j-\frac
12},t\big)\right)^{\frac
 12},\\
 &&e_{q,n}=\max_{j\in\bZ_{N+1}}\left|(q-\hat{q}_h)(x_{j-\frac 12},t)\right|,
\;\; \|e_q\|_{*}=\left(\frac
1{N+1}\sum_{j=1}^{N+1}\big(q-\hat{q}_h\big)^2\big(x_{j-\frac
12},t\big)\right)^{\frac
 12},
\end{eqnarray*}
  with the numerical fluxes $\hat{u}_h,\hat{q}_h$ taken as
  \eqref{flux1} or \eqref{flux2}.
\end{theorem}
\begin{proof}
   Let $(u_I,q_I)=(u^k_I,q^k_I)$. By \eqref{interp:1} and
   \eqref{interp:2},
\[
   u(x_{j-\frac 12},t)=\hat{u}_I(x_{j-\frac 12},t), \ \ q(x_{j-\frac 12},t)=\hat{q}_I(x_{j-\frac
   12},t),\ \ j\in\bZ_{N+1}.
\]
  For any fixed $t$, $u_I-u_h\in \bP_k$ in each $\tau_j,j\in \bZ_N$. Then the inverse inequality holds and thus,
\begin{eqnarray}\label{e5}
   \left|(\hat{u}_I-\hat{u}_h)(x_{j+\frac
   12},t)\right|&\le &\|u_I-u_h\|_{0,\infty,\Omega_j}(t)\lesssim h^{-\frac 12}\|u_I-u_h\|_{0,\Omega_j}(t),\\
   \label{e6}
 \left|(\hat{q}_I-\hat{q}_h)(x_{j+\frac
   12},t)\right|&\le &\|q_I-q_h\|_{0,\infty,\Omega_j}(t)\lesssim h^{-\frac
   12}\|q_I-q_h\|_{0,\Omega_j}(t).
\end{eqnarray}
  Here $\Omega_j=\tau_j\cup\tau_{j+1}, j\in\bZ_{N-1}$ and
  $\Omega_j=\tau_1\cup\tau_N, j=0,N$. By \eqref{spclossness},
  the desired result \eqref{super_node2} follows.

We next show \eqref{super_node3}. Again by the inverse inequality,
\begin{eqnarray*}
\frac {1}{N}\sum_{j=1}^{N}\|v\|^2_{0,\infty,\tau_j}\lesssim \frac
{1}{N}\sum_{j=1}^{N} h_j^{-1}\|v\|^2_{0,\tau_j}
 \lesssim \|v\|_0^2,\ \ \forall v\in V_h.
\end{eqnarray*}
 Then
\begin{eqnarray*}
    \frac 1{N+1}\sum_{j=1}^{N+1}\big(\hat{u}_I-\hat{u}_h\big)^2\big(x_{j-\frac
    12},t\big)&\lesssim& \|u_I-u_h\|_0^2(t),\\
   \frac 1{N+1}\sum_{j=1}^{N+1}\big(\hat{q}_I-\hat{q}_h\big)^2\big(x_{j-\frac
    12},t\big)&\lesssim& \|q_I-q_h\|_0^2(t).
\end{eqnarray*}
  The inequality \eqref{super_node3} follows directly from the
  estimate \eqref{spclossness}.
\end{proof}

\subsection{ Superconvegence for the domain and cell averages}
  We first denote by $\|e_u\|_d$ and $\|e_u\|_c$ the domain average
  and the cell average of $u-u_h$, respectively. Precisely,
\begin{eqnarray*}
   &&\|e_u\|_d=
   \left|\frac{1}{2\pi}\int_{0}^{2\pi}(u-u_h)(x,t)dx\right|,\\
  &&\|e_u\|_c=\left( \frac{1}{N}\sum_{j=1}^N\Big(\frac{1}{h_j}\int_{x_{j-\frac 12}}^{x_{j+\frac 12}}(u-u_h)(x,t)dx\Big)^2\right)^{\frac
  12}.
\end{eqnarray*}
  Similarly, the domain average $\|e_q\|_d$ and the cell average
  $\|e_q\|_c$ of  $q-q_h$ can be defined as the same way.

 We have the following superconvergence results for the domain and
 cell averages.
\begin{theorem}\label{theorem3}
     Suppose all the conditions of Theorem \ref{theo:1} hold. Then
\begin{equation}\label{cell-average}
     \|e_u\|_c\lesssim (h+t^{\frac 32}+t) h^{2k}\|u\|_{2k+2,\infty},\ \
   \|e_q\|_c\lesssim (1+t)h^{2k}\|u\|_{2k+2,\infty}.
\end{equation}
  In addition, there hold, for the periodic boundary condition
\begin{equation}\label{cell-average1}
  \|e_u\|_d\lesssim  h^{2k+1}\|u\|_{2k+2,\infty},\ \ \|e_q\|_d=0,
\end{equation}
  and for the mixed  boundary condition
\begin{equation}\label{cell-average3}
   \|e_u\|_d\lesssim
   (h^{\frac 12}+t^{\frac 32}+t)h^{2k+\frac 12}\|u\|_{2k+2,\infty},\ \
   \|e_q\|_d\lesssim
   (1+t)h^{2k+\frac 12}\|u\|_{2k+2,\infty}.
\end{equation}

\end{theorem}
\begin{proof} Note that
$a^i_j(u-u_h,q-q_h;v)=0,\forall v\in V_h, i=1,2, j\in\bZ_N$. By
taking $v=1$, we obtain
\begin{eqnarray*}
   &&\int_{x_{j-\frac 12}}^{x_{j+\frac
   12}}(q-q_h)(x,t)dx=(u-\hat{u}_h)(x_{j+\frac 12},t)-(u-\hat{u}_h)(x_{j-\frac
   12},t),\\
   &&\int_{x_{j-\frac 12}}^{x_{j+\frac
   12}}(u-u_h)_t(x,t)dx=(q-\hat{q}_h)(x_{j+\frac 12},t)-(q-\hat{q}_h)(x_{j-\frac
   12},t).
\end{eqnarray*}
  In light of \eqref{e5}-\eqref{e6},
\begin{eqnarray*}
    &&\left|\int_{x_{j-\frac 12}}^{x_{j+\frac
   12}}(q-q_h)(x,t)dx\right|\lesssim h^{-\frac
   12}\|u_I^k-u_h\|_{0,\Omega_j},\\
    &&\left|\int_{x_{j-\frac 12}}^{x_{j+\frac
   12}}(u-u_h)_t(x,t)dx\right|\lesssim h^{-\frac
   12}\|q_I^k-q_h\|_{0,\Omega_j}.
\end{eqnarray*}
   Then
\[
   \|e_q\|_c\lesssim \left(\frac {1}{N}\sum_{j=1}^N h^{-3}\|u_I^k-u_h\|^2_{0,\Omega_j}\right)^{\frac
   12}\lesssim h^{-1}\|u_I^k-u_h\|_0.
\]
  The second inequality of \eqref{cell-average} follows
  directly from the estimate \eqref{spclossness}. On the other hand, since
  \begin{eqnarray*}
     \int_{x_{j-\frac 12}}^{x_{j+\frac 12}}(u-u_h)(x,t)dx&=& \int_{x_{j-\frac 12}}^{x_{j+\frac 12}}(u-u_h)(x,0)dx+\int_0^t\frac{d}{dt}
   \int_{x_{j-\frac 12}}^{x_{j+\frac 12}}(u-u_h)(x,t)dxdt\\
   &\lesssim& \int_{x_{j-\frac 12}}^{x_{j+\frac 12}}(u-u_h)(x,0)dx+th^{-\frac
   12}\|q_I^k-q_I\|_{0,\Omega_j},
\end{eqnarray*}
   then the estimate for the
  cell average of $u-u_h$ at $\tau_j,j\in\bZ_N$ at any time $t>0$ is reduced to the estimate at $t=0$.
   By the special
    initial condition,
\[
   \int_{x_{j-\frac 12}}^{x_{j+\frac 12}}(u-u_h)(x,0)dx=\int_{x_{j-\frac 12}}^{x_{j+\frac 12}}(u-u_I^k)(x,0)dx
   =\int_{x_{j-\frac 12}}^{x_{j+\frac 12}}W^k_2(x,0)dx,
\]
  where $W^k_2$ is defined by \eqref{corr_func} for fluxes \eqref{flux1}, or \eqref{f1} for fluxes \eqref{flux2}.
  When $W^k_2$ is defined by  \eqref{corr_func}, we have, from
   \eqref{fp1}-\eqref{fp2}, \eqref{fp3}-\eqref{fp4} and the orthogonal properties of Legendre polynomials,
\begin{eqnarray*}
    &&\int_{x_{j-\frac 12}}^{x_{j+\frac 12}} W^k_2(x,t) dx=\int_{x_{j-\frac 12}}^{x_{j+\frac 12}} \bar w_{1,r}(x,t)
    dx,\ \ k=2r,\\
 &&\int_{x_{j-\frac 12}}^{x_{j+\frac 12}} W^k_2(x,t) dx=\int_{x_{j-\frac
12}}^{x_{j+\frac 12}} w_{2,r+1}(x,t) dx,\ \ k=2r+1.
\end{eqnarray*}
  Recall the estimates for $\bar w_{1,r}$ and $\bar w_{2,r+1}$ in \eqref{e:1}-\eqref{e:2},  we obtain for all $k\ge 1$,
\[
   \left|\int_{x_{j-\frac 12}}^{x_{j+\frac 12}} W^k_2(x,t)
   dx\right|\lesssim h^{2k+2}\|u\|_{2k+2,\infty,\tau_j},\ \ \forall
   j\in\bZ_N.
\]
 Similarly, when $W^k_2$ is defined by \eqref{f1}, the above inequality still holds true. Then,
\begin{equation}\label{initial_est}
   \left| \int_{x_{j-\frac 12}}^{x_{j+\frac 12}} (u-u_h)(x,0)dx\right|\lesssim
   h^{2k+2}\|u\|_{2k+2,\infty,\tau_j},
\end{equation}
   which yields
\[
  \left|\int_{x_{j-\frac 12}}^{x_{j+\frac 12}}(u-u_h)(x,t)dx\right|
   \lesssim h^{2k+2}\|u\|_{2k+2,\infty,\tau_j}+th^{-\frac
   12}\|q_I^k-q_I\|_{0,\Omega_j}.
\]
   Then a direct calculation and the estimate \eqref{spclossness}
   yield the first inequality of \eqref{cell-average}.

    Now we move on to the domain average. Noticing  that
\begin{eqnarray*}
   \int_{0}^{2\pi}(q-q_h)(x,t)dx
   &=&(u-\hat{u}_h)(x_{N+\frac 12},t)-(u-\hat{u}_h)(x_{\frac
   12},t),\\
\int_{0}^{2\pi}(u-u_h)_t(x,t)dx
   &=&(q-\hat{q}_h)(x_{N+\frac 12},t)-(q-\hat{q}_h)(x_{\frac
   12},t),
\end{eqnarray*}
  the second inequalities of \eqref{cell-average1} and  \eqref{cell-average3} follow from
  the fact $(u-\hat{u}_h)(x_{N+\frac 12},t)=(u-\hat{u}_h)(x_{\frac 12},t)$
  for the
  periodic boundary condition  and \eqref{super_node2}
  for the mixed  boundary condition, respectively. As for the
   domain average of $u-u_h$,  by
\eqref{super_node2}, the fact that
 $(q-\hat{q}_h)(x_{N+\frac 12},t)=(q-\hat{q}_h)(x_{\frac 12},t)$ for
 the
 periodic boundary condition, we obtain for the periodic boundary condition
\[
    \int_{0}^{2\pi}(u-u_h)(x,t)dx=\int_{0}^{2\pi}(u-u_h)(x,0)dx,
\]
    and for the mixed  boundary condition
\[
    \left|\int_{0}^{2\pi}(u-u_h)(x,t)dx\right|\lesssim
    \left|\int_{0}^{2\pi}(u-u_h)(x,0)dx\right| + (t^{\frac 32}+t)h^{2k+\frac
   12}\|u\|_{2k+2,\infty}.
\]
   In light of \eqref{initial_est}, the first inequalities of \eqref{cell-average1} and \eqref{cell-average3} follow.
\end{proof}

\subsection{Superconvergence of the function value approximation at Radau points}
  As a by-product of \eqref{spclossness}, we have the following superconvergence
  results of the function value approximation at Radau points.

\begin{theorem}\label{theo:3}
   Suppose all the conditions of Corollary \ref{coro:1} hold.  For both the periodic and mixed  boundary
   conditions, there hold,
\begin{eqnarray}\label{radau5}
   e_{u,r}\lesssim
   (1+t\sqrt{h})h^{k+2}\|u\|_{k+4,\infty},\ \ e_{q,l}\lesssim
  (1+\sqrt{th})h^{k+2}\|u\|_{k+4,\infty}
\end{eqnarray}
  for fluxes \eqref{flux1} and
\begin{eqnarray}\label{radu7}
   e_{u,l}\lesssim
   (1+t\sqrt{h})h^{k+2}\|u\|_{k+4,\infty},\ \
   e_{q,r}\lesssim
   (1+\sqrt{th})h^{k+2}\|u\|_{k+4,\infty}
\end{eqnarray}
  for fluxes \eqref{flux2}. Here
\begin{eqnarray*}
   &&e_{u,r}=\max_{(j,m)\in\bZ_N\times\bZ_k}\left|(u-u_h)(R^r_{j,m},t)\right|,\
   e_{u,l}=\max_{(j,m)\in\bZ_N\times\bZ_k}\left|(u-u_h)(R^l_{j,m},t)\right|,\\
   &&e_{q,r}=\max_{(j,m)\in\bZ_N\times\bZ_k}\left|(q-q_h)(R^r_{j,m},t)\right|,\ \
   e_{q,l}=\max_{(j,m)\in\bZ_N\times\bZ_k}\left|(q-q_h)(R^l_{j,m},t)\right|.
\end{eqnarray*}
\end{theorem}
\begin{proof}
  We first consider \eqref{radau5}.  By using the inverse
inequality and choosing $l=2$ in \eqref{spclossness}, we obtain
\begin{eqnarray*}
   &&\|u^2_I-u_h\|_{0,\infty}\lesssim h^{-\frac 12}\|u^2_I-u_h\|_0\lesssim
  (1+t) h^{k+\frac 52}\|u\|_{k+4,\infty},\\
 &&\|q^2_I-q_h\|_{0,\infty}\lesssim h^{-\frac
 12}\|q^2_I-q_h\|_0\lesssim (1+t^{\frac 12})
   h^{k+\frac 52}\|u\|_{k+4,\infty}.
\end{eqnarray*}
   By \eqref{estimate:w} and the triangular inequality,
 \begin{eqnarray}\label{r5}
   &&\|u_h-P_h^-u\|_{0,\infty} \lesssim \|W_2^2\|_{0,\infty}+\|u^2_I-u_h\|_{0,\infty}\lesssim (1+t\sqrt{h})h^{k+2}\|u\|_{k+4,\infty},\\
   \label{r6}
 &&\|q_h-P_h^+q\|_{0,\infty} \lesssim \|W^2_1\|_{0,\infty}+\|q^2_I-q_h\|_{0,\infty}\lesssim
 (1+\sqrt{th})h^{k+2}\|u\|_{k+4,\infty}.
\end{eqnarray}
   For all $v\in W^{k+2,\infty}(\Omega)$, the standard approximation
   theory gives
\[
  \left|(v-P_h^-v)(R^r_{j,m},t)\right|\lesssim
   h^{k+2}\|v\|_{k+2,\infty},\ \  \left|(v-P_h^+v)(R^l_{j,m},t)\right|\lesssim
    h^{k+2}\|v\|_{k+2,\infty}.
\]
   Then \eqref{radau5} follows. The proof of \eqref{radu7}
   can be obtained by the same arguments.
\end{proof}

\subsection{Superconvergence of the derivative approximation at Radau points}

   For all $v\in W^{k+2,\infty}(\Omega)$, it is shown in
   \cite{Cao;zhang;zou:2k+1} that
\begin{equation}\label{r7}
\left|\big(v_x-(P_h^-v)_x\big)(R^l_{j,m},t)\right|\lesssim
  h^{k+1}\|v\|_{k+2,\infty},\ \ \forall (j,m)\in\bZ_{N}\times\bZ_k.
\end{equation}
 Similarly, we can obtain
\begin{equation}\label{r8}
\left|\big(v_x-(P_h^+v)_x\big)(R^r_{j,m},t)\right|\lesssim
  h^{k+1}\|v\|_{k+2,\infty},\ \ \forall (j,m)\in\bZ_{N}\times\bZ_k.
\end{equation}

  We have the following superconvergence  results.
\begin{theorem}\label{theo:2}
   Suppose all the conditions of Corollary \ref{coro:1} hold. Let
\begin{eqnarray*}
 &&  e_{ux,l}=\max_{j,m}\left|(u_x-u_{hx})(R^l_{j,m},t)\right|,\
   \
   e_{ux,r}=\max_{j,m}\left|(u_x-u_{hx})(R^r_{j,m},t)\right|,\\
&&
  e_{qx,l}=\max_{j,m}\left|(q_x-q_{hx})(R^l_{j,m},t)\right|,\
   \
   e_{qx,r}=\max_{j,m}\left|(q_x-q_{hx})(R^r_{j,m},t)\right|.
\end{eqnarray*}
  For both the periodic and mixed  boundary
   conditions, there hold,
\begin{eqnarray}\label{radau_2}
    e_{ux,l}
  \lesssim
   (1+t\sqrt{h})h^{k+1}\|u\|_{k+4,\infty},\ \ e_{qx,r}\lesssim
   (1+\sqrt{th})h^{k+1}\|u\|_{k+4,\infty}
\end{eqnarray}
  for fluxes \eqref{flux1} and
\begin{eqnarray}\label{radu4}
  e_{ux,r}\lesssim
   (1+t\sqrt{h})h^{k+1}\|u\|_{k+4,\infty},\ \ e_{qx,l}\lesssim
   (1+\sqrt{th})h^{k+1}\|u\|_{k+4,\infty}
\end{eqnarray}
  for fluxes \eqref{flux2}.
\end{theorem}
\begin{proof}
    Using  the inverse inequality in \eqref{r5}-\eqref{r6} gives
\begin{eqnarray*}
   && |P_h^-u-u_h|_{1,\infty}\lesssim (1+t\sqrt{h})h^{k+1}\|u\|_{k+4,\infty},\\
   && |P_h^+q-q_h|_{1,\infty}\lesssim
   (1+\sqrt{th})h^{k+1}\|u\|_{k+4,\infty}.
\end{eqnarray*}
  Then the desired result
  \eqref{radau_2} follows from \eqref{r7}-\eqref{r8} and the triangular inequality.
  The proof of \eqref{radu4} is following the
  same line.
\end{proof}

To end  this section, we would like to demonstrate how to calculate
$u^l_I(x,0),1\le l\le k$
 only using the information of the initial value $u_0(x)$. Without loss
 of generality, we consider the fluxes choice \eqref{flux1}.
Since $u_t=u_{xx}$, we have for all integers $i\ge 1$
\[
\frac{\partial^i}{\partial t^i}u(x,0)= u_0^{(2i)}(x), \quad \forall
x\in\Omega.
\]
Therefore, by \eqref{coefficient}, we have the derivatives at $t=0$,
for all $1\le i \le \lceil { k/2 }\rceil$
\begin{eqnarray}\label{derivm_u}
&&\bar{u}^{(i)}_{j,k}=-u_0^{(2i)}(x^-_{j+\frac
12})+\frac{1}{h_j}\int_{\tau_j}
u_0^{(2i)}\sum_{m=0}^k(2m+1)L_{j,m},\\
\label{derivm_q}
&&\tilde{q}^{(i)}_{j,k}=(-1)^{k+1}u_0^{(2i+1)}(x^+_{j-\frac
12})+\frac{1}{h_j}\int_{\tau_j}
u_0^{(2i+1)}\sum_{m=0}^{k}(-1)^{k+m}(2m+1)L_{j,m}.
\end{eqnarray}

Now we divide the process into the following steps :
\begin{itemize}
    \item[1.] In each element of $\tau_j$, calculate
    $G_i=\bar{u}^{(i)}_{j,k}, Q_i=\tilde{q}^{(i)}_{j,k}$ by
    \eqref{derivm_u}-\eqref{derivm_q}.
    \item[2.] Compute $F_{1,i},F_{2,i}$ from \eqref{correcti} and \eqref{correctii}.
    \item[3.] Calculate $\bar{F}_{1,i}$ by $F_{1,i}$ and \eqref{barF}.
    \item[4.] Choose
    $\bar{w}_{1,i}=h^{2i}G_i\bar{F}_{1,i},{w}_{2,i}=h^{2i-1}Q_{i-1}{F}_{2,i}$
    and $w^l=\sum\limits_{i=1}^{\lfloor l/2\rfloor}\bar{w}_{1,i}+\sum\limits_{i=1}^{\lceil
    l/2\rceil}{w}_{2,i}$.
   \item[4.] Figure out $u^l_I=P^-_hu_0-w^l$.
\end{itemize}

\section{Numerical results}
   In this section, we present numerical examples to verify our
   theoretical findings. We shall measure  various norms, including
   $\xi_u,\xi_q$, the numerical fluxes at nodes, interior left and right Radau
   points, and the domain and cell averages, which are defined in Corollary \ref{coro:1}
   and Theorems \ref{theo:1}-\ref{theo:2}, respectively.

{\it Example} 1. We consider the following problem
\begin{eqnarray*}
\begin{aligned}
   &u_t=u_{xx},\ \ &&(x,t)\in [0,2\pi]\times(0, 1], \\
   &u(x,0)=\sin(x),\ \ && x\in [0,2\pi]
\end{aligned}
\end{eqnarray*}
  with periodic boundary condition $u(0,t)=u(2\pi,t)$. The exact solution is
\[
   u(x,t)=e^{-t}\sin(x).
\]

  We solve this problem by the LDG scheme \eqref{LDG_scheme} with $k=3,4$,
  respectively. The numerical fluxes are chosen as \eqref{flux1}, and
  the initial solution $u_h(x,0)=u_I^k(x,0)$ with $u_I^k$ defined by
  \eqref{interpolation}. We construct our meshes by equally dividing
  each interval,  $[0,\frac{3\pi}{4}]$ and $[\frac{3\pi}{4},2\pi]$, into ${N}/{2}$ subintervals,  $N=2^m$, $m=2,3,\ldots,7$.
  To reduce the time discretization error, we use the ninth order strong-stability preserving (SSP) Runge-Kutta  method \cite{Gottlieb;Shu:SIAMSSP}
  with time step $\triangle t=0.01h^2_{min}$
   in $k=3$ and $\triangle t=0.001h^2_{min}$ in $k=4$, where $h_{min}={3\pi}/{2N}$.

Numerical data are demonstrated in Tables \ref{T1}-\ref{T2}, and
corresponding error curves are depicted in Figures \ref{F1}-\ref{F2}
on the log-log scale.

\begin{table}[htbp]\caption{  Various errors in the periodic boundary condition for $k=3$.}\label{T1} \centering
\begin{threeparttable}
        \begin{tabular}{c c c c c c c c   }
        \hline
     N     & $\|\xi_u\|_0$ & $e_{u,r}$ &  $e_{ux,l}$ & $e_{u,n}$& $\|e_u\|_*$  & $\|e_u\|_c$ & $\|e_u\|_d$   \\
        \hline\cline{1-8}
 4  & 5.06e-04  & 2.82e-04  & 1.72e-04 &  1.04e-04 &  5.92e-05 &  7.57e-05  & 3.12e-04 \\
  8  & 1.29e-05  & 5.13e-06  & 4.14e-06 &  7.19e-07 &  4.17e-07 &  5.11e-07  & 2.00e-06\\
 16  & 3.92e-07  & 1.35e-07 &  1.25e-07 &  5.65e-09 &  3.16e-09 &  3.89e-09  &1.44e-08\\
 32  & 1.22e-08 &  3.98e-09 &  3.92e-09 &   4.37e-11&   2.44e-11&   3.01e-11 & 1.08e-10\\
 64  & 3.80e-10 &  1.24e-10 &  1.23e-10 &  3.40e-13 &  1.89e-13 &  2.34e-13  &8.30e-13\\
 128 &  1.19e-11&   3.85e-12&   3.85e-12&   2.65e-15 & 1.47e-15& 1.83e-15  & 6.43e-15\\
\hline\cline{1-8}
 N     & $\|\xi_q\|_0$ & $e_{q,l}$ &  $e_{qx,r}$ & $e_{q,n}$& $\|e_q\|_*$  & $\|e_q\|_c$ & $\|e_q\|_d$   \\
        \hline\cline{1-8}
 4 &   4.01e-04 &  1.72e-04  & 2.41e-04 &  1.12e-04 &  6.67e-05 &  3.74e-05  & 1.84e-13\\
 8  &   1.11e-05 &  4.14e-06 &  4.76e-06 &  6.86e-07 &  4.12e-07 &  3.18e-07 &  4.05e-17\\
 16 &  3.44e-07 &  1.25e-07 &  1.32e-07 &  5.24e-09 &  3.17e-09 &  2.56e-09 &  9.64e-21\\
 32 &  1.07e-08  & 3.92e-09 &  3.98e-09 &  4.14e-11 &  2.47e-11 &  2.02e-11  & 2.34e-24\\
 64  &   3.35e-10  & 1.23e-10 &  1.24e-10 &  3.24e-13 &  1.93e-13 &  1.59e-13 &  5.70e-28\\
 128  &  1.05e-11  & 3.85e-12 &  3.85e-12 &  2.54e-15 &  1.51e-15 &  1.24e-15 &  1.39e-31\\
 \hline
       \end{tabular}
 \end{threeparttable}
\end{table}

\begin{table}[htbp]\caption{  Various errors in the periodic boundary condition for  $k=4$ }\label{T2} \centering
\begin{threeparttable}
        \begin{tabular}{c c c c c c c c   }
        \hline
   N     & $\|\xi_u\|_0$ & $e_{u,r}$ &  $e_{ux,l}$ & $e_{u,n}$& $\|e_u\|_*$  & $\|e_u\|_c$ & $\|e_u\|_d$   \\
        \hline\cline{1-8}
 4 &   2.34e-05 &  8.40e-06  &  1.06e-05   & 1.47e-06   & 7.96e-07    &1.04e-06    &2.94e-06 \\
 8  &  3.92e-07  &  1.32e-07  &  1.28e-07  &  2.62e-09  &  1.47e-09   & 1.74e-09  &  5.96e-09 \\
 16 &   6.21e-09  &  2.02e-09 &   2.05e-09  &  5.06e-12   & 2.85e-12  &  3.37e-12  &  1.21e-11 \\
32  &   9.73e-11  &  3.18e-11  &  3.18e-11  &  1.01e-14   & 5.59e-15   & 6.60e-15  &  2.41e-14 \\
64  &   1.52e-12  &  4.97e-13  &  4.97e-13  &  1.97e-17   & 1.09e-17   & 1.29e-17  &  4.77e-17 \\
128 &    2.38e-14  &  7.76e-15  &  7.76e-15  &  3.85e-20   & 2.14e-20   & 2.53e-20  &  9.37e-20 \\
 \hline\cline{1-8}
 N     & $\|\xi_q\|_0$ & $e_{q,l}$ &  $e_{qx,r}$ & $e_{q,n}$& $\|e_q\|_*$  & $\|e_q\|_c$ & $\|e_q\|_d$   \\
        \hline\cline{1-8}
 4 &  3.06e-05 &  1.07e-05 &  9.30e-06 &  1.92e-06 &  1.14e-06  & 6.52e-07 &  1.86e-13\\
 8 &  4.56e-07 &  1.28e-07  & 1.33e-07 &  2.41e-09 & 1.50e-09   &1.23e-09  & 4.06e-17\\
 16 & 7.09e-09 &  2.05e-09  & 2.03e-09 &  4.44e-12 &  2.78e-12  & 2.37e-12 &  9.64e-21\\
 32 &  1.11e-10 &  3.18e-11  & 3.18e-11 &  8.55e-15 &  5.36e-15  & 4.60e-15 &  2.34e-24\\
 64 &  1.73e-12 &  4.97e-13  & 4.96e-13 &  1.66e-17 &  1.04e-17  & 8.97e-18 &  5.70e-28\\
128 &   2.70e-14 &  7.76e-15  & 7.76e-15 &  3.25e-20 &2.03e-20  & 1.75e-20 &  1.39e-31 \\
 \hline
\end{tabular}
 \end{threeparttable}
\end{table}

\begin{figure}[htbp]
\scalebox{0.5}{\includegraphics{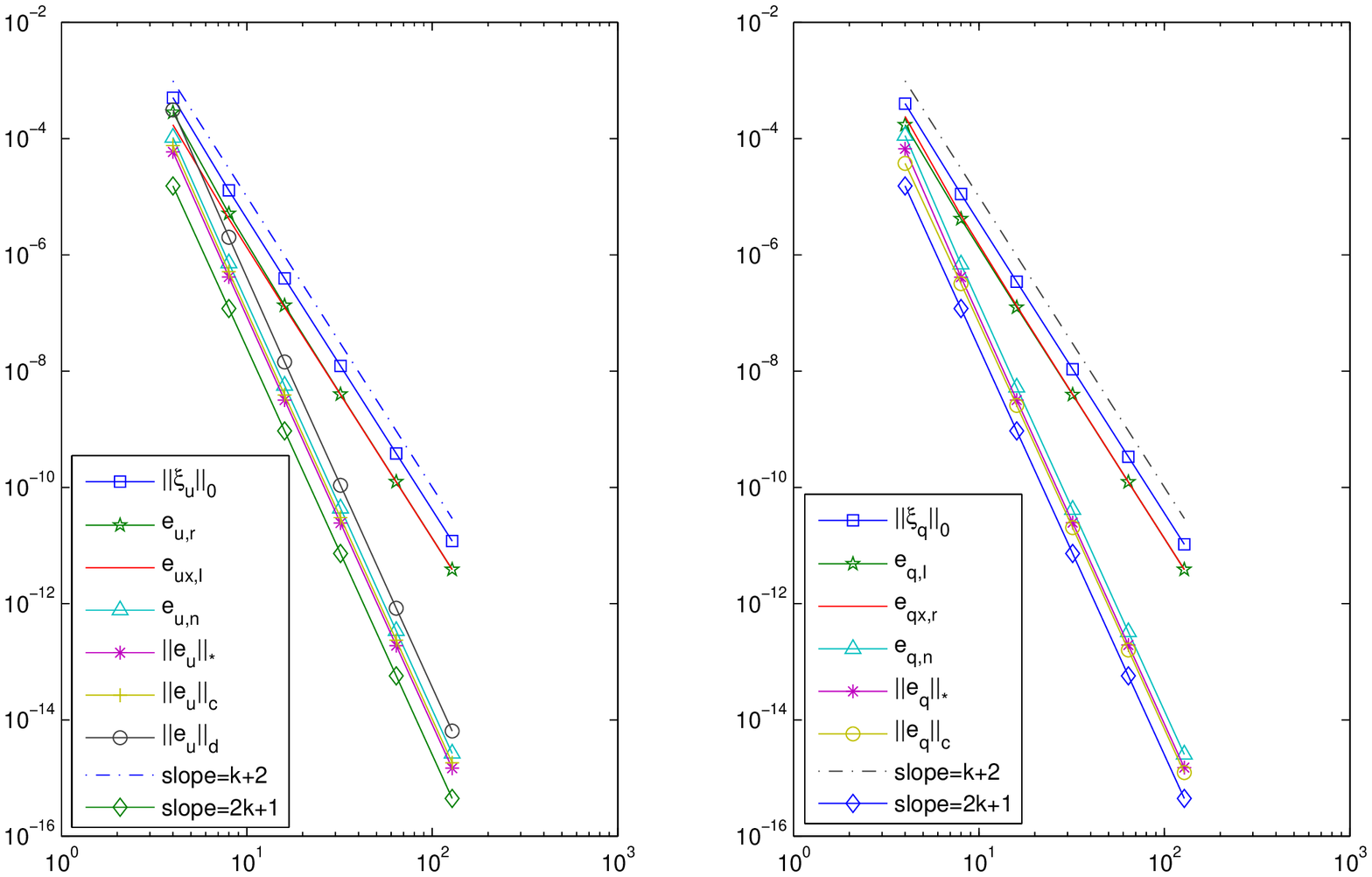}}
 \caption{Error curves in the periodic boundary condition for $k=3$.}\label{F1}
\end{figure}

\begin{figure}[htbp]
\scalebox{0.5}{\includegraphics{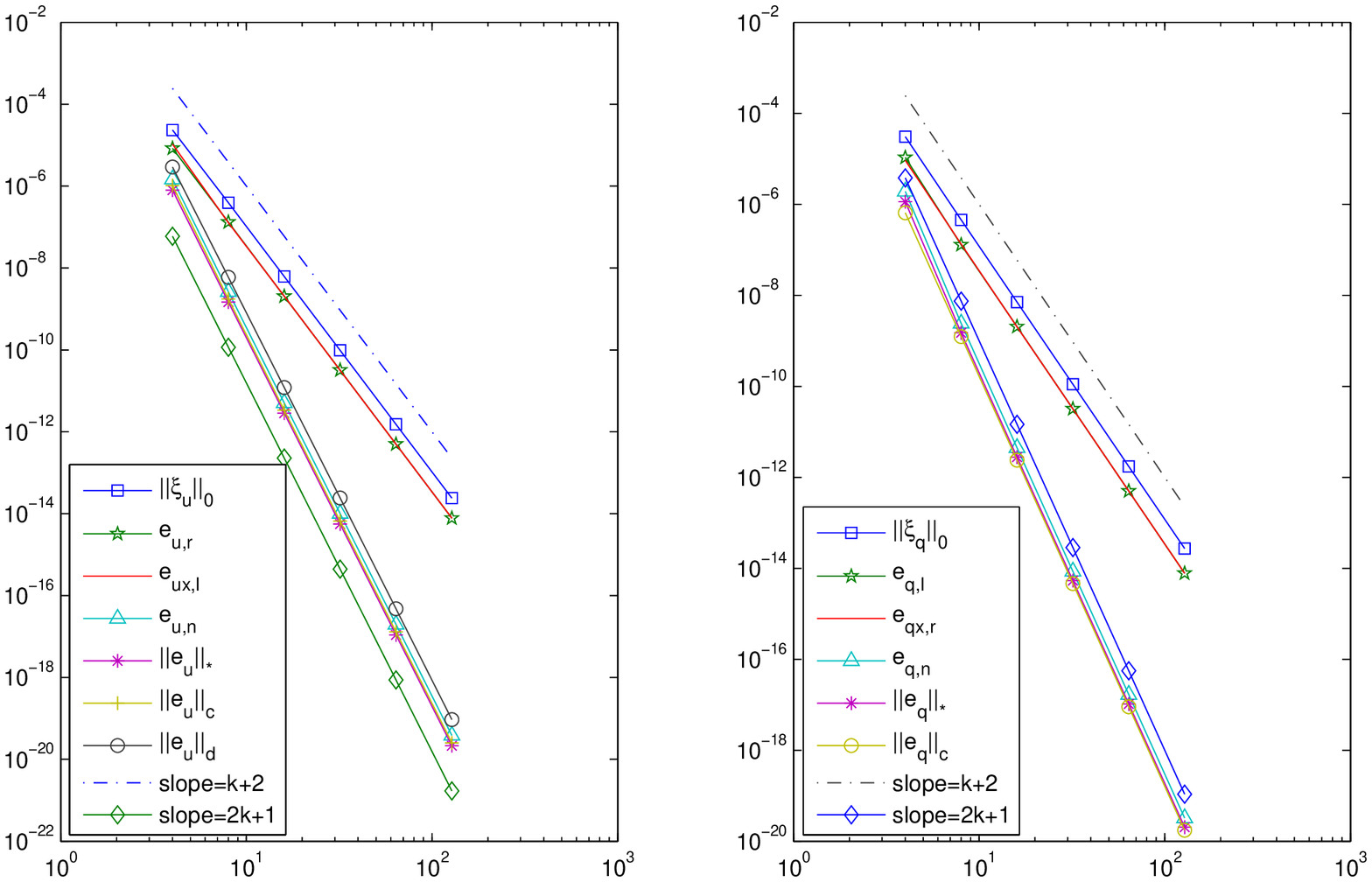}}
 \caption{Error curves in the periodic boundary condition for $k=4$}\label{F2}
\end{figure}

 We observe from  Figures \ref{F1}-\ref{F2} a convergence rate $k+2$ for
 $\|\xi_u\|_0, \|\xi_q\|_0$, $e_{u,r},e_{q,l}$, $e_{ux,l}$ and $e_{qx,r}$, and $2k+1$ for
 $\|e_u\|_*, \|e_q\|_*$  and $ e_{u,n},e_{q,n}$. These results
 confirm our theoretical findings in Corollary \ref{coro:1}, Theorem \ref{theo:1}, and Theorems \ref{theo:3}-\ref{theo:2} :
  for fluxes choice \eqref{flux1}, the LDG solution $(u_h,q_h)$ is
  $k+2$th order superconvergent to the Gauss-Radau projection of the
  exact solution $(P_h^-u, P_h^+q)$;
  the function value error $u-u_h$ at right Radau points and its derivative error $u'-u_h' $ at interior left Radau points,
  and $q-q_h$ at left Radau points and $q'-q_h'$ at interior right Radau points, all converge with the same rate $k+2$;
  the maximum and average errors of $u-u_h$ and $q-q_h$
  are supercovergent at downwind points and upwind points, respectively, with  the same rate $2k+1$.
  Moreover, our numerical results demonstrate that the superconvergence
  rates in \eqref{xi:1}, \eqref{super_node3} and \eqref{radau5} are optimal;
 while the convergence rate for the derivative
approximation at
 Radau points are one order better than the estimate provided in
 \eqref{radau_2}.

 For the domain and cell averages, we first observe, from Tables
 \ref{T1}-\ref{T2}, that the error for the domain average of
  $q-q_h$ reaches the machine precision at the initial mesh, which
  indicates  the equality in
  \eqref{cell-average1} is true. Then from Figures
  \ref{F1}-\ref{F2}, we observe a $2k+1$th superconvergence rate for
  the domain average of $u-u_h$, as predicted in
  \eqref{cell-average1}. Furthermore,
   we also observe $2k+1$th superconvergence
 rates for the cell average of $u-u_h$ and $q-q_h$, one order higher
  than the one given in \eqref{cell-average}.

 {\it Example} 2. We consider the following problem
\begin{eqnarray*}
\begin{aligned}
   &u_t+u_x=0,\ \ &&(x,t)\in [0,2\pi]\times(0,1], \\
   &u(x,0)=\sin(x)
\end{aligned}
\end{eqnarray*}
  with mixed  boundary condition
\[
   u_x(0,0)=e^{t+1},\ \ \  u(2\pi,0)=e^{-t}+e^{2\pi+t+1}.
\]
  The exact solution to this problem is
\[
   u(x,t)=e^{-t}\cos(x)+e^{x+t+1}.
\]

  The problem is solved by the LDG
 scheme \eqref{LDG_scheme} with $k=3,4$, respectively.
The numerical fluxes are chosen as \eqref{flux2}, and
  the initial solution $u_h(x,0)=u_I^k(x,0)$ with $u_I^k$ defined by
  \eqref{interpolation1}.  Uniform meshes are used, which are constructed by
  dividing the interval $[0,2\pi]$ into  $N = 2^m$ ($m=2,3,\ldots,6$) equal subintervals. The fourth order
  Runge-Kutta method is used to diminish the time discretization error
  with time step $\triangle t=T/n$  using $n=1000N^2$
  in $k=3$, and $n=5000N^2$ in $k=4$.

 Listed in Tables \ref{T3}-\ref{T4} are numerical data for  various errors
 in cases $k=3,4$. Depicted in Figures \ref{F3}-\ref{F4} are corresponding
  error curves with log-log scale.

 Again, we observe similar superconvergence  phenomena as in the
 periodic case.  To be more precise, if we choose the numerical fluxes
 \eqref{flux2},  the LDG solution $u_h$
   converges to the Gauss-Radau projection $P_h^+u$ with a rate of $k+2$, as well as the derivative approximation at
  all interior right Radau points and the function value approximation at all left Radau points;
  as for the domain and cell averages, along with the maximum and average errors at upwinding points, the convergent rate is
  $2k+1$; while for the solution $q_h$, it is
   convergent to the Gauss-Radau projection $P_h^-q$ with a rate of $k+2$, the same rate for the derivative approximation at
    all interior left Radau points and the function value approximation at all right Radau points;
    finally, convergence rates of the maximum and average
  errors at downwind points as well as the domain and cell averages are all $2k+1$. These results
 confirm our theoretical findings in Corollary \ref{coro:1}, Theorems
 \ref{theo:1}-\ref{theo:2}.
  Note that the $2k+1$th superconvergence rate
  for the domain average is $ 1/2$ order
 higher than the one given in \eqref{cell-average3},  and the $k+2$th
 superconvergence rate for the derivative approximation is one order
 better than the estimate provided in \eqref{radu4}.

\begin{table}[htbp]\caption{  Various errors in the  mixed  boundary condition for $k=3$.}\label{T3}\centering
\begin{threeparttable}
        \begin{tabular}{c c c c c c c c   }
        \hline
  N     & $\|\xi_u\|_0$ & $e_{u,l}$ &  $e_{ux,r}$ & $e_{u,n}$& $\|e_u\|_*$  & $\|e_u\|_c$ & $\|e_u\|_d$   \\
\hline\cline{1-8}
4 &  5.08e-01 &  2.63e-01  & 2.33e-01 &  5.50e-03  & 3.64e-03  & 1.46e-02  & 6.67e-02\\
8 &  1.89e-02 &  1.06e-02 &  1.05e-02 &  2.57e-05  & 1.63e-05  & 1.36e-04  & 5.59e-04\\
16&   6.28e-04 &  3.73e-04 &  3.85e-04 &  1.62e-07 &  9.89e-08 &  1.13e-06 &  4.46e-06\\
32&   2.00e-05 &  1.24e-05 &  1.29e-05 &  1.16e-09 &  7.00e-10 &  8.99e-09 & 3.49e-08\\
64 &  6.31e-07 &  4.07e-07 &  4.19e-07 &  8.80e-12 &  5.25e-12 &  7.05e-11 &  2.73e-10\\
\hline\cline{1-8}
  N     & $\|\xi_q\|_0$ & $e_{q,r}$ &  $e_{qx,l}$ & $e_{q,n}$& $\|e_q\|_*$  & $\|e_q\|_c$ & $\|e_q\|_d$   \\
 \hline\cline{1-8}
4  & 5.99e-01   & 2.33e-01   & 2.62e-01  &  4.48e-02   & 2.30e-02   & 2.02e-03   & 7.30e-04 \\
8   & 2.07e-02  &  1.05e-02  &  1.06e-02  &  4.68e-04  &  1.85e-04  &  8.39e-06  &  4.90e-06\\
16  &  6.57e-04  &  3.85e-04  &  3.73e-04 &   3.95e-06  &  1.32e-06  &  5.22e-08  &  3.62e-08\\
32  &  2.05e-05  &  1.29e-05  &  1.24e-05  &  3.14e-08  &  9.63e-09  &  3.88e-10  &  2.76e-10\\
64  &  6.38e-07  &  4.19e-07  &  4.07e-07  &  2.47e-10  &  7.21e-11   & 3.00e-12  &  2.14e-12\\
\hline
       \end{tabular}
 \end{threeparttable}
\end{table}

\begin{table}[htbp]\caption{  Various errors in the mixed  boundary condition for $k=4$. }\label{T4} \centering
\begin{threeparttable}
        \begin{tabular}{c c c c c c c c   }
        \hline
  N     & $\|\xi_u\|_0$ & $e_{u,l}$ &  $e_{ux,r}$ & $e_{u,n}$& $\|e_u\|_*$  & $\|e_u\|_c$ & $\|e_u\|_d$   \\
        \hline\cline{1-8}
4 & 3.05e-02 &  1.14e-02&  1.12e-02 &  2.59e-05  & 1.61e-05 &  1.03e-04 &   4.55e-04 \\
8 &  5.61e-04 &  2.40e-04&   2.47e-04 &  3.60e-08 &  2.12e-08 &  2.55e-07 &  1.01e-06\\
16 &  9.23e-06&   4.67e-06&   4.72e-06 &  6.42e-11&   3.77e-11 &  5.36e-10 &  2.07e-09\\
32 &  1.47e-07&   8.09e-08 &  8.14e-08 &  1.27e-13 &  7.36e-14 &  1.07e-12 &  4.10e-12\\
64 &  2.31e-09&   1.33e-09 &  1.34e-09 &  2.50e-16 &  1.46e-16 &  2.10e-15 &  8.06e-15\\
 \hline\cline{1-8}
  N     & $\|\xi_q\|_0$ & $e_{q,r}$ &  $e_{qx,l}$ & $e_{q,n}$& $\|e_q\|_*$  & $\|e_q\|_c$ & $\|e_q\|_d$   \\
        \hline\cline{1-8}
4  & 3.46e-02 &  1.12e-02 & 1.14e-02 & 3.64e-04 & 1.85e-04 &  8.32e-06 &  3.43e-06\\
8  & 5.99e-04 &  2.47e-04 & 2.40e-04 & 9.08e-07 &  3.57e-07&  1.34e-08 &  6.60e-09\\
16 & 9.54e-06 &  4.72e-06 & 4.67e-06 & 1.89e-09 &  6.33e-10&  2.42e-11 &  1.30e-11\\
32 & 1.49e-07 &  8.14e-08 & 8.09e-08 & 3.75e-12 &  1.15e-12&  4.63e-14 &  2.58e-14\\
64 & 2.33e-09 &  1.34e-09 & 1.33e-09 & 7.64e-15 &  2.16e-15&  8.99e-17 &  5.07e-17\\
 \hline
 \end{tabular}
 \end{threeparttable}
\end{table}

\begin{figure}[htbp]
\scalebox{0.5}{\includegraphics{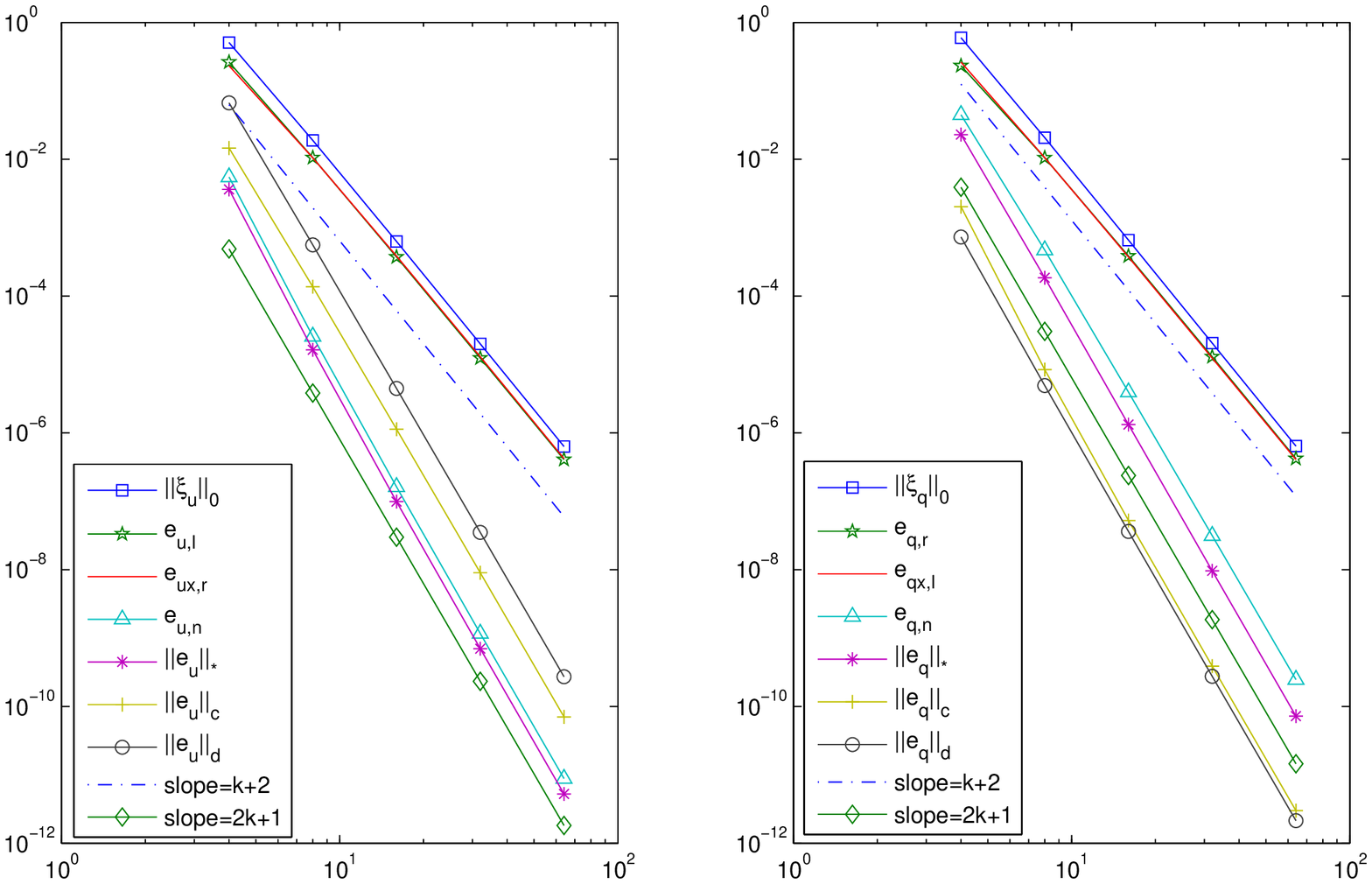}}
 \caption{ Error curves in the mixed  boundary condition for $k=3$.}\label{F3}
\end{figure}

\begin{figure}[htbp]
\scalebox{0.5}{\includegraphics{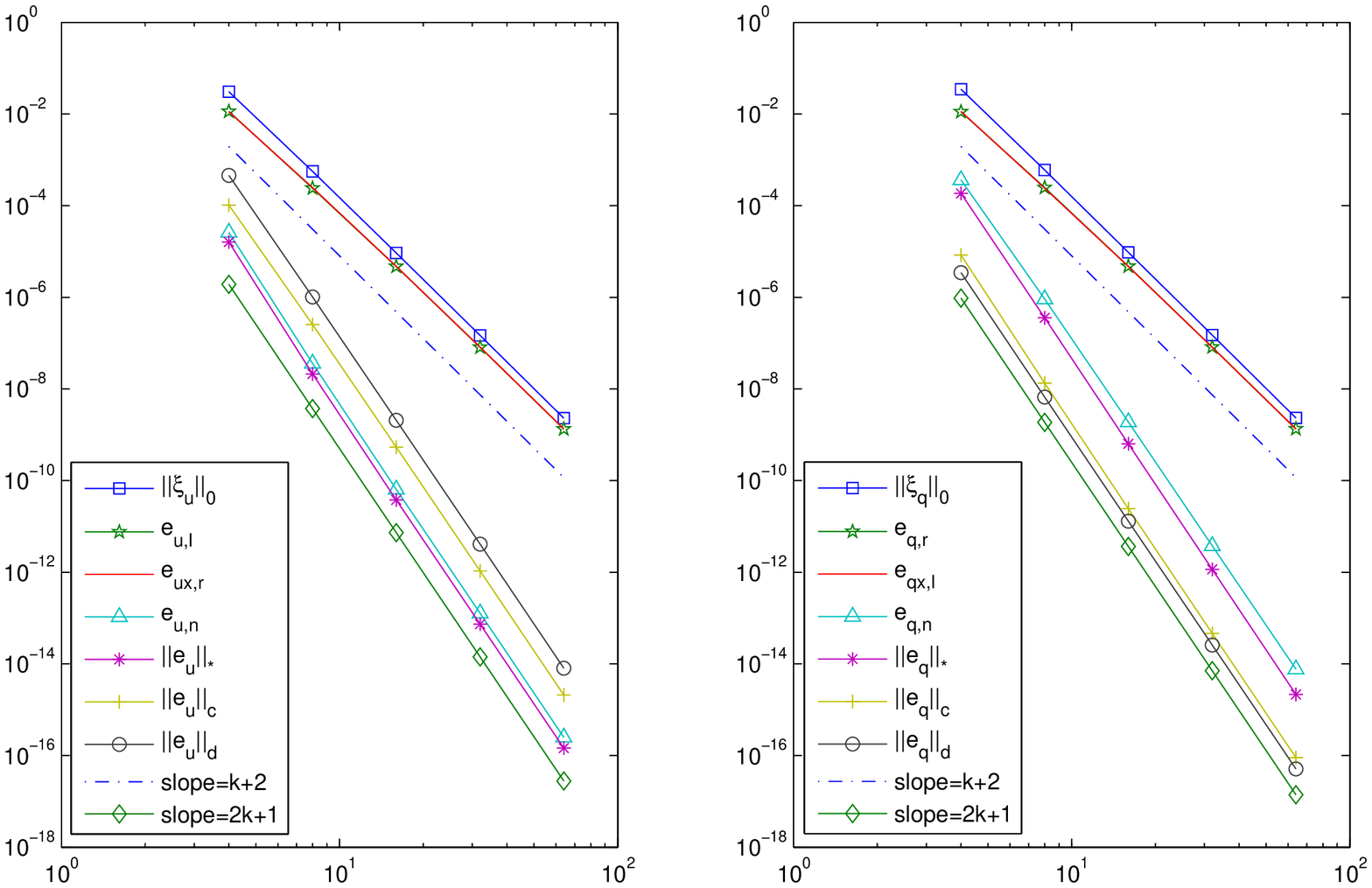}}
 \caption{Error curves in the  mixed  boundary condition for $k=4$.}\label{F4}
\end{figure}

\section{Concluding remarks}
  To summarize, we have established a $2k+1$th
superconvergence rate for the domain average and numerical fluxes at
all nodes (on average). As a direct consequence, we obtain a
$k+1$th superconvergence rate for the derivative approximation
   and a $k+2$th superconvergence rate for the function value approximation of the LDG solution at the Radau points.
   In addition, we also prove that the LDG
   solution is superconvergent with a $k+2$th rate to the Gauss-Radau projection of the exact solution, and a $2k$th
    rate to the exact solution in the cell average sense. Numerical test data demonstrates that
    most of our error bounds are sharp, and to the best of our knowledge, the $k+2$th derivative superconvergence rate at the Radau points
    is reported for the first time in the literature.
    Our current and
    future works include convection-diffusion equations and 2-D problems, which would be more challenging and interesting.

\end{document}